\documentclass[12pt,reqno]{amsart}
\usepackage{amsmath,amssymb,extarrows}
\usepackage{url}
\usepackage{tikz,enumerate}
\usepackage{diagbox}
\usepackage{appendix}
\usepackage{epic}

\usepackage{float} 

\usepackage{comment}

\usepackage{cite}

\usepackage{hyperref}
\usepackage{array}

\usepackage{booktabs}
\usepackage{multirow}

\allowdisplaybreaks[4]

\newtheorem{theorem}{Theorem}
\newtheorem{corollary}[theorem]{Corollary}
\newtheorem{conj}[theorem]{Conjecture}
\newtheorem{lemma}[theorem]{Lemma}

\theoremstyle{definition}

\newtheorem{problem}[theorem]{Problem}
\theoremstyle{remark}
\newtheorem{rem}{Remark}
\numberwithin{equation}{section}
\numberwithin{theorem}{section}
\numberwithin{defn}{section}

\begin{document}
\title[Identities for Rank Two Partial Nahm Sums]{Rogers--Ramanujan Type Identities for Rank Two Partial Nahm Sums}

\author{Liuquan Wang and Wentao Zeng}
\address{School of Mathematics and Statistics, Wuhan University, Wuhan 430072, Hubei, People's Republic of China}
\email{wanglq@whu.edu.cn;mathlqwang@163.com}

\address{School of Mathematics and Statistics, Wuhan University, Wuhan 430072, Hubei, People's Republic of China}
\email{zengwentao2000@whu.edu.cn}



\begin{abstract}
Let $A$ be a $r\times r$ rational nonzero symmetric matrix, $B$ a rational column vector, $C$ a rational scalar. For any integer lattice $L$ and vector $v$ of $\mathbb{Z}^r$, we define Nahm sum on the lattice coset $v+L\in \mathbb{Z}^r/L$:
\begin{align*}\label{eq-lattice-sum}
f_{A,B,C,v+L}(q):=\sum_{n=(n_1,\dots,n_r)^\mathrm{T} \in v+L} \frac{q^{\frac{1}{2}n^\mathrm{T} An+n^\mathrm{T} B+C}}{(q;q)_{n_1}\cdots (q;q)_{n_r}}.
\end{align*}
If $L$ is a full rank lattice and a proper subset of $\mathbb{Z}^r$, then we call $f_{A,B,C,v+L}(q)$ a rank $r$ partial Nahm sum. When the rank $r=1$, we find eight modular partial Nahm sums using some known identities. When the rank $r=2$ and $L$ is one of the lattices $\mathbb{Z}(2,0)+\mathbb{Z}(0,1)$, $\mathbb{Z}(1,0)+\mathbb{Z}(0,2)$ or $\mathbb{Z}(2,0)+\mathbb{Z}(0,2)$, we find 14 types of symmetric matrices $A$ such that there exist vectors $B,v$ and scalars $C$ so that the partial Nahm sum $f_{A,B,C,v+L}(q)$ is modular. We establish Rogers--Ramanujan type identities for the corresponding partial Nahm sums which prove their modularity.
\end{abstract}

\maketitle

\section{Introduction}\label{sec-intro}
We begin with some $q$-series notations:
\begin{align}
    (a;q)_\infty&:=\prod\limits_{k=0}^{\infty }(1-aq^k), \quad |q|<1, \\
    (a;q)_n&:=\frac{(a;q)_\infty}{(aq^n;q)_\infty}, \quad n\in \mathbb{R}.
\end{align}
Clearly for any positive integer $n$ we have
\begin{align}
    (a;q)_n=(1-a)(1-aq)\cdots (1-aq^{n-1}). \label{aq-defn}
\end{align}
Sometime we will used compressed notation:
\begin{align}
    (a_1,a_2,\dots,a_m;q)_n:=\prod\limits_{k=1}^m (a_k;q)_n.
\end{align}

Rogers--Ramanujan type identities are certain sum-to-product identities in which the left side is a mixed sum of some $q$-hypergeometric series and the right sides are some infinite $q$-products. The study of them is inspired by the famous Rogers--Ramanujan identities:
\begin{align}\label{RR}
\sum_{n=0}^\infty \frac{q^{n^2+in}}{(q;q)_n}=\frac{1}{(q^{i+1},q^{4-i};q^5)_\infty}, \quad i\in \{0,1\}.
\end{align}
These identities were first proved by Rogers \cite{Rogers1894} and later rediscovered by Ramanujan. As a multi-sum generalization of it, the Andrews--Gordon identities \cite{Andrews1974,Gordon1961}  assert that for integer $k\geq 2$ and $1\leq i \leq k$,
\begin{align}
\sum_{n_1,\dots,n_{k-1}\geq 0} \frac{q^{N_1^2+\cdots+N_{k-1}^2+N_i+\cdots +N_{k-1}}}{(q;q)_{n_1}\cdots (q;q)_{n_{k-2}} (q;q)_{n_{k-1}}} =\frac{(q^i,q^{2k+1-i},q^{2k+1};q^{2k+1})_\infty}{(q;q)_\infty},\label{AG}
\end{align}
where $N_j=n_j+\cdots+n_{k-1}$ if $j\leq k-1$ and $N_k=0$.

Many Rogers--Ramanujan type identities have been discovered so far and a famous work on this topic is Slater's list \cite{Slater} which contains 130 such identities. We refer the reader to Sills' book \cite{Sills-book} for detailed introduction on this topic.

Rogers--Ramanujan type identities have important applications in various branches of mathematics including number theory, combinatorics, Lie algebra and mathematical physics, etc. Nahm's problem reveal the important connections among Rogers--Ramanujan type identities, modular forms and conformal field theory. The problem is to find all positive definite $r\times r$ rational matrix $A$, $r$-dimensional rational column vector $B$ and rational scalar $C$ such that the Nahm sum
\begin{align}\label{eq-full-Nahm}
f_{A,B,C}(q):=\sum_{n=(n_1,\dots,n_r)^\mathrm{T} \in \mathbb{Z}^r} \frac{q^{\frac{1}{2}n^\mathrm{T} An+n^\mathrm{T} B+C}}{(q;q)_{n_1}\cdots (q;q)_{n_r}}
\end{align}
is modular. In view of the convention that
\begin{align}
    {1}/{(q;q)_n}=0, \quad n \in \mathbb{Z}_{<0}, \label{q-zero}
\end{align}
we see that the sum is actually over the vectors $n \in \mathbb{Z}_{\geq 0}^r$. Modular Nahm sums are usually expected to be characters of some 2-dimensional rational conformal field theories.

Around 2007, Zagier \cite{Zagier} solved the rank one case of Nahm's problem. He proved that there are exactly seven rank one modular Nahm sums and they correspond to
\begin{align}
(A,B,C)=&(2,0,-1/60), ~~ (2,1,11/60), ~~ (1/2,0,-1/40),(1/2,1/2,1/40),  \nonumber \\
&(1,-1/2,1/24), ~~ (1,0,-1/48), ~~(1,1/2,1/24).
\end{align}
The modularity of the first two triples are justified by the identities \eqref{RR}. The third and fourth triples follow from Rogers' identities \cite{Rogers1894}:
\begin{align}
    &\sum_{n=0}^\infty \frac{q^{n^2}}{(q^4;q^4)_n} =\frac{1}{(-q^2;q^2)_\infty (q,q^4;q^{5})_\infty}, \label{Slater20} \\
&\sum_{n=0}^\infty \frac{q^{n(n+2)}}{(q^4;q^4)_n}=\frac{1}{(-q^2;q^2)_\infty (q^2,q^{3};q^{5})_\infty}. \label{Slater16}
\end{align}
The modularity of the last three triples are justified by the case $z=1,q^{1/2},q$ in one of Euler's $q$-exponential identities \cite[Corollary 2.2]{Andrews-book}:
\begin{align}\label{euler-2}
   \sum_{n=0}^\infty \frac{q^{\frac{1}{2}(n^2-n)}z^n}{(q;q)_n}= (-z;q)_\infty.
\end{align}

Zagier \cite[Tables 2 and 3]{Zagier} also discovered two lists of Nahm sums of rank two and three which are conjecturally modular. Their modularity have now all been confirmed by works of Zagier \cite{Zagier}, Cherednik--Feigin \cite{Feigin}, Vlasenko--Zwegers \cite{VZ}, Wang \cite{Wang-rank2,Wang-rank3} and Cao--Rosengren--Wang \cite{Cao-Rosengren-Wang}.

In this paper, instead of the full Nahm sum in \eqref{eq-full-Nahm}, we are interested in \emph{partial Nahm sums}  restricted to some subset of $\mathbb{Z}^r$. To be specific, Let $A$ be a $r\times r$  nonzero rational symmetric matrix, $B$ a $r$-dimensional column vector and $C$ a scalar. For any lattice $L\subseteq \mathbb{Z}^r$ and coset $v+L\in \mathbb{Z}^r/L$, we define \emph{Nahm sum on lattice coset} as
\begin{align}\label{eq-lattice-sum}
f_{A,B,C,v+L}(q):=\sum_{n=(n_1,\dots,n_r)^\mathrm{T} \in v+L} \frac{q^{\frac{1}{2}n^\mathrm{T} An+n^\mathrm{T} B+C}}{(q;q)_{n_1}\cdots (q;q)_{n_r}}.
\end{align}
Again, the sum is actually taken over the set $(v+L)\cap \mathbb{Z}_{\geq 0}^r$. In particular, when $L=\mathbb{Z}^r$, we have $f_{A,B,C,v+L}(q)=f_{A,B,C}(q)$. As a natural generalization to Nahm's problem, we propose the following problem.
\begin{problem}
Find all symmetric nonzero matrix $A\in \mathbb{Q}^{r\times r}$, column vector $B\in \mathbb{Q}^r$, scalar $C\in \mathbb{Q}$, lattice $L\subseteq \mathbb{Z}^r$ and coset $v+L \in \mathbb{Z}^r/L$ such that the sum in \eqref{eq-lattice-sum} converge absolutely and $f_{A,B,C,v+L}(q)$ is modular.
\end{problem}
Note that we no longer require $A$ to be positive definite, and we call $f_{A,B,C}(q)$ with $A$ symmetric a \emph{generalized Nahm sum}. For convenience, we call $(A,B,C,v+L)$ a \emph{modular quadruple} when $f_{A,B,C,v+L}(q)$ is modular. We are more interested in the case when $L$ is of full rank $r$ and a proper subset of $\mathbb{Z}^r$, and in this case we call $f_{A,B,C,v+L}(q)$ a rank $r$ \emph{partial Nahm sum}.

In the rank one case, from some known single-sum Rogers--Ramanujan type identities in Slater's list (see \eqref{s-39}--\eqref{s-96}), we immediately obtain the following examples.
\begin{theorem}\label{thm-rank1}
For $v\in \{0,1\}$, the partial Nahm sum $f_{A,B,C,v+L}(q)$ are modular for $(A,B,C,v+L)$ being
\begin{equation}
    \begin{split}
   &(1,0,-1/48,v+2\mathbb{Z}),~~(1,-1/2,1/24,v+2\mathbb{Z}), \\ &(1/2,0, -1/40, v+2\mathbb{Z}), ~~ (1/2,1/2,1/40,v+2\mathbb{Z}).
\end{split}
\end{equation}
\end{theorem}

This paper aims to provide some rank two modular partial Nahm sums on lattice cosets.  For convenience, we write the matrix as
$$A=M(a,b,c):=\begin{pmatrix}
    a & b \\ b & c
\end{pmatrix}$$
and rewrite the indices $n_1,n_2$ as $i,j$. The quadratic form appearing in the exponent of $q$ in the Nahm sum $f_{A,B,C,v+L}(q)$ is
$$\frac{1}{2}n^\mathrm{T}A n=\frac{1}{2}ai^2+bij+\frac{1}{2}cj^2.$$
The determinant of the matrix is $\det A=ac-b^2$. If $b=0$, then $A$ is diagonal and the Nahm sum $f_{A,B,C,v+L}(q)$ can be decomposed into product of two single sums directly. Hence we will only consider the more interesting case $b\neq 0$. We choose the lattice $L$ as
\begin{align}\label{eq-lattice}
\mathbb{Z}(2,0)^\mathrm{T}+\mathbb{Z}(0,1)^\mathrm{T}, ~~ \mathbb{Z}(1,0)^\mathrm{T}+\mathbb{Z}(0,2)^\mathrm{T}, ~~ \mathbb{Z}(2,0)^\mathrm{T}+\mathbb{Z}(0,2)^\mathrm{T}
\end{align}
and consider partial Nahm sums on all cosets of $\mathbb{Z}^2/L$. Interchanging $n_1$ with $n_2$ we get Nahm sums for the matrix $M(c,b,a)$, and hence we list $M(a,b,c)$ and $M(c,b,a)$ together but discuss only one of them. Through an extensive search by Maple, we find 14 sets of partial Nahm sums which are likely to be expressed as modular infinite products, and we confirm their modularity by establishing the corresponding Rogers--Ramanujan type identities.

We give two examples to illustrate our main results. For the matrix $A=M(0,1,0)$  we find that the partial Nahm sum $f_{A,B,C,v+L}(q)$ is modular for
\begin{equation}
\begin{split}
B=(1/2,1/2)^\mathrm{T}, ~~ C=-5/12,
 ~~ L=\mathbb{Z}(2,0)^\mathrm{T}+\mathbb{Z}(0,2)^\mathrm{T}
\end{split}
\end{equation}
and any vector $v\in \mathbb{Z}^2$. This is justified by the following identity.
\begin{theorem}\label{thm-010}
    We have
    \begin{align}
       \sum_{i,j\ge0}\frac{q^{4ij+i+3j}}{(q;q)_{2i+1}(q;q)_{2j}}&=\frac{(q^8;q^8)_\infty}{(q;q^2)_\infty^2(q^4;q^8)_\infty}. \label{eq3-2}
    \end{align}
\end{theorem}

For the matrix $A=M(3/4,-1/4,3/4)$, Vlasenko and Zwegers \cite[Theorem 3.2]{VZ} proved that $f_{A,B_i,C_i}(q)$ ($i=1,2,3,4$) is modular for
\begin{equation}
    \begin{split}
    &B_1=(1/4,-1/4)^\mathrm{T}, ~ B_2=(-1/4,1/4)^\mathrm{T}, ~ B_3=(1/2,0)^\mathrm{T}, ~ B_4=(0,1/2)^\mathrm{T}, \\
    &C_1=C_2=-1/80, \quad C_3=C_4=1/80.
\end{split}
\end{equation}
We find that some partial Nahm sums for the same matrix are also modular because of the following result .
\begin{theorem}\label{thm-last}
    We have
    \begin{align}
        &\sum_{i,j\ge0}\frac{q^{\frac{3}{2}i^2-ij+\frac{3}{2}j^2-\frac{1}{2}i+\frac{1}{2}j}}{(q;q)_{2i}(q;q)_{2j}}=\frac{(-q^3,-q^5,q^8;q^8)_\infty}{(q;q)_\infty (q^2,q^8;q^{10})_\infty},\label{eq10-1}\\
         &\sum_{i,j\ge0}\frac{q^{\frac{3}{2}i^2-ij+\frac{3}{2}j^2+\frac{1}{2}i+\frac{3}{2}j}}{(q;q)_{2i+1}(q;q)_{2j+1}}=\frac{(-q,-q^7,q^8;q^8)_\infty}{(q;q)_\infty (q^2,q^8;q^{10})_\infty}, \label{eq10-4} \\
        &\sum_{i,j\ge0}\frac{q^{\frac{3}{2}i^2-ij+\frac{3}{2}j^2+\frac{3}{2}i+\frac{1}{2}j}}{(q;q)_{2i+1}(q;q)_{2j}}=\frac{(-q^3,-q^5,q^8;q^8)_\infty}{(q;q)_\infty (q^4,q^6;q^{10})_\infty},\label{eq10-2}\\
        &\sum_{i,j\ge0}\frac{q^{\frac{3}{2}i^2-ij+\frac{3}{2}j^2-\frac{1}{2}i+\frac{5}{2}j}}{(q;q)_{2i}(q;q)_{2j+1}}=\frac{(-q,-q^7,q^8;q^8)_\infty}{(q;q)_\infty (q^4,q^6;q^{10})_\infty}. \label{eq10-3}
    \end{align}
\end{theorem}
As a consequence, $f_{A,B,C,v+L}(q)$ is modular for $(A,B,C,v+L)$ given in Table \ref{tab-14}.
\begin{table}[H]
\caption{Modular quadruples $(A,B,C,v+L)$ with $A=M(3/4,-1/4,3/4)$} \label{tab-14}
\centering
\renewcommand{\arraystretch}{1.5}
\begin{tabular}{|c|c|c|c|}
  \hline
  $B$ & $C$ & $v$ & $L$ \\ \hline
  $(-1/4,1/4)^\mathrm{T}$,  $(1/4,-1/4)^\mathrm{T}$   & $-1/80$  & $(0,0)^\mathrm{T}$,  $(1,1)^\mathrm{T}$  & \multirow{2}{*}{$\mathbb{Z}(2,0)^\mathrm{T}+\mathbb{Z}(0,2)^\mathrm{T}$} \\ \cline{1-3}
   $(0,1/2)^\mathrm{T}$, $(1/2,0)^\mathrm{T}$  &  $1/80$ & $(1,0)^\mathrm{T}$, $(0,1)^\mathrm{T}$ & \\
   \hline
\end{tabular}
\end{table}

The rest of this paper is organized as follows. We collect some auxiliary identities in Section \ref{sec-pre} and discuss some sufficient conditions such that the Nahm sum converge absolutely. We present identities for modular Nahm sums on lattice cosets for matrices with negative, zero and positive determinant in Sections \ref{sec-negative}--\ref{sec-positive}, respectively. Finally, in Section \ref{sec-remarks} we briefly discuss several problems to be considered in the future.

\section{Preliminaries}\label{sec-pre}
The following identity \cite[Corollary 2.2]{Andrews-book} and \eqref{euler-2} are usually referred as Euler's exponential identities:
\begin{align}
    \sum_{n=0}^{\infty} \frac{z^{n}}{(q;q)_{n}}=\frac{1}{(z;q)_{\infty}}, \quad|z|<1.\label{euler-1}
\end{align}
Both of them are consequences of the $q$-binomial theorem \cite[Theorem 2.1]{Andrews-book}:
\begin{align}
    \sum_{n=0}^{\infty}&\frac{(a;q)_{n}z^{n}}{(q;q)_{n}}=\frac{(az;q)_{\infty}}{(z;q)_{\infty}} \quad (|z|<1). \label{q-binomial}
\end{align}

We define the $q$-binomial coefficient or Gaussian coefficient
\begin{align*}
    {n \brack m } = {n \brack m }_{q} := \begin{cases}
        \frac{\displaystyle(q;q)_{n}}{\displaystyle(q;q)_{m} \displaystyle(q;q)_{n-m}}, & 0\le m \le n \\
        0,  & \text{otherwise}
        \end{cases}.
\end{align*}
Now we can write a finite version of \eqref{euler-2} \cite[Theorem 3.3]{Andrews-book}:
\begin{align}
    \sum_{i=0}^{n}{n\brack i}z^iq^{\frac{i^2-i}{2}}=(-z;q)_{n}. \label{euler-finite}
\end{align}

Recall the Jacobi triple product identity \cite[Theorem 2.8]{Andrews-book}:
\begin{align}
    \sum_{n=-\infty }^{\infty}q^{(n^2-n)/2}z^{n}=(-z,-q/z,q;q)_{\infty} \quad (z\ne 0).\label{Jacobi}
\end{align}
As an application of \eqref{Jacobi}, we have
\begin{align}
        &(-q^{a+b},-q^{a-b},q^{2a};q^{2a})_\infty+(q^{a+b},q^{a-b},q^{2a};q^{2a})_\infty=\sum_{n=-\infty}^{\infty}\big( q^{an^2+bn}+(-1)^n q^{an^2+bn}\big) \notag\\
        &=2\sum_{n=-\infty}^{\infty} q^{a(2n)^2+b(2n)}=2(-q^{4a+2b},-q^{4a-2b},q^{8a};q^{8a})_{\infty}, \label{jac even}  \\
        &(-q^{a+b},-q^{a-b},q^{2a};q^{2a})_\infty-(q^{a+b},q^{a-b},q^{2a};q^{2a})_\infty=\sum_{n=-\infty}^{\infty}\big( q^{an^2+bn}-(-1)^n q^{an^2+bn}\big) \notag\\
        &=2\sum_{n=-\infty}^{\infty} q^{a(2n-1)^2+b(2n-1)}=2q^{a-b}(-q^{2b},-q^{8a-2b},q^{8a};q^{8a})_{\infty}. \label{jac odd}
    \end{align}
In particular, when $(a,b)=(2,1)$, we deduce from \eqref{jac even} and \eqref{jac odd} that
    \begin{align}
            (-q;q^2)_\infty+(q;q^2)_\infty&=\frac{2(-q^{6},-q^{10},q^{16};q^{16})_\infty}{(q^4;q^4)_\infty},\label{ss-1}\\
            (-q;q^2)_\infty-(q;q^2)_\infty&=\frac{2q(-q^{2},-q^{14},q^{16};q^{16})_\infty}{(q^4;q^4)_\infty}.\label{ss-2}
     \end{align}

For $n\ge 0$ we have
\begin{align}
    (-aq^{-n};q)_{\infty}&=a^{n}q^{-\frac{1}{2}n^2-\frac{1}{2}n}(-a^{-1}q;q)_{n}(-a;q)_{\infty}\label{trans-g}.
\end{align}
This will be used frequently but usually without mention.

We will invoke the following known single-sum Rogers--Ramanujan type identities:
\begin{align}
  & \sum_{n=0}^{\infty}\frac{q^{2n^2}}{(q;q)_{2n}}=\frac{(-q^{3},-q^{5},q^{8};q^{8})_{\infty}}{(q^{2};q^{2})_{\infty}},\quad\text{(S.\ 39, S.\ 83)}\label{s-39}\\
        & \sum_{n=0}^{\infty}\frac{q^{2n(n+1)}}{(q;q)_{2n+1}}=\frac{(-q,-q^{7},q^{8};q^{8})_{\infty}}{(q^{2};q^{2})_{\infty}},\quad\text{(S.\ 38, S.\ 86)}\label{s-38}\\
  &\sum_{n=0}^\infty \frac{q^{2n^2-n}}{(q;q)_{2n}}=(-q;q)_\infty, \quad \text{(S.\ 85)} \label{S85} \\
       &  \sum_{n=0}^\infty \frac{q^{2n^2+n}}{(q;q)_{2n+1}}=(-q;q)_\infty, \quad \text{(S.\ 9, S.\ 84)} \label{S9} \\
       &  \sum_{n=0}^{\infty}\frac{q^{n^2}}{(q;q)_{2n}}=\frac{1}{(q;q^2)_\infty(q^{4},q^{16};q^{20})_{\infty}},\quad\text{(\cite[p.\ 330 (3), 1st Eq.]{Rogers1894}; S.\ 79, S.\ 98)}\label{s-79}\\
          & \sum_{n=0}^{\infty}\frac{q^{n(n+1)}}{(q;q)_{2n+1}}=\frac{(q^{3},q^{7},q^{10};q^{10})_{\infty}(q^{4},q^{16};q^{20})_{\infty}}{(q;q)_\infty}, \quad\text{(S.\ 94)}\label{s-94}\\
       &  \sum_{n=0}^{\infty}\frac{q^{n(n+1)}}{(q;q)_{2n}}=\frac{(q,q^{9},q^{10};q^{10})_{\infty}(q^{8},q^{12};q^{20})_{\infty}}{(q;q)_\infty}, \quad\text{(S.\ 99)}\label{s-99}\\
        & \sum_{n=0}^{\infty}\frac{q^{n(n+2)}}{(q;q)_{2n+1}}=\frac{1}{(q;q^2)_\infty(q^{8},q^{12};q^{20})_{\infty}}, \quad\text{(\cite[p.\ 330 (3), 2nd Eq.]{Rogers1917}, S.\ 96)}\label{s-96} \\
       & \sum_{n=0}^{\infty}\frac{q^{2n(n+1)}}{(q^2;q^2)_{n}(-q;q)_{2n+1}}=\sum_{n=0}^{\infty}\frac{q^{2n(n+1)}(q;q^2)_{n+1}}{(q^2;q^2)_{2n+1}}=\frac{(q,q^6,q^7;q^7)_\infty}{(q^2;q^2)_\infty},\quad\text{(S.\ 31)}\label{s-31}\\
      &  \sum_{n=0}^{\infty}\frac{q^{2n(n+1)}}{(q^2;q^2)_{n}(-q;q)_{2n}}=\sum_{n=0}^{\infty}\frac{q^{2n(n+1)}(q;q^2)_{n}}{(q^2;q^2)_{2n}}=\frac{(q^2,q^5,q^7;q^7)_\infty}{(q^2;q^2)_\infty},\quad\text{(S.\ 32)}\label{s-32}\\
       & \sum_{n=0}^{\infty}\frac{q^{2n^2}}{(q^2;q^2)_{n}(-q;q)_{2n}}=\sum_{n=0}^{\infty}\frac{q^{2n^2}(q;q^2)_{n}}{(q^2;q^2)_{2n}}=\frac{(q^3,q^4,q^7;q^7)_\infty}{(q^2;q^2)_\infty},\quad\text{(S.\ 33)}\label{s-33}\\
     &   \sum_{n=0}^{\infty}\frac{q^{n(3n+1)/2}(-q;q)_{n}}{(q;q)_{2n+1}}=\frac{(q^{4},q^{6},q^{10};q^{10})_{\infty}}{(q;q)_\infty},\quad\text{(S.\ 62)}\label{s-62}\\
       & \sum_{n=0}^{\infty}\frac{q^{n(3n-1)/2}(-q;q)_{n}}{(q;q)_{2n}}=\frac{(q^{4},q^{6},q^{10};q^{10})_{\infty}}{(q;q)_\infty},\quad\text{(S.\ 46)}\label{s-46}\\
      &  \sum_{n=0}^{\infty}\frac{q^{3n(n+1)/2}(-q;q)_{n}}{(q;q)_{2n+1}}=\frac{(q^{2},q^{8},q^{10};q^{10})_{\infty}}{(q;q)_\infty},\quad\text{(S.\ 63)}\label{s-63}\\
        & \sum_{n=0}^{\infty}\frac{q^{3n^2}(-q;q^2)_{n}}{(q^2;q^2)_{2n}}=\frac{(q^{10};q^{20})_{\infty}}{(q^{3},q^{4},q^{5},q^{7},q^{13},q^{15},q^{16},q^{17};q^{20})_{\infty}}, \quad\text{(S.\ 100c)}\label{s-100c}\\
       &  \sum_{n=0}^{\infty}\frac{q^{n(3n-2)}(-q;q^2)_{n}}{(q^2;q^2)_{2n}}=\frac{(q^{10};q^{20})_{\infty}}{(q,q^{5},q^{8},q^{9},q^{11},q^{12},q^{15},q^{19};q^{20})_{\infty}},  \quad\text{(S.\ 95)}\label{s-95} \\
   & \sum_{n=0}^{\infty}\frac{q^{n^2-n}z^{n}}{(q;q)_{n}(z;q)_{n}}=\frac{1}{(z;q)_{\infty}}. \quad \text{( \cite[Corollary 2.7]{Andrews-book})} \label{cauchy}
\end{align}
Here the label S.\ $n$ denotes the equation $(n)$ in Slater's list \cite{Slater}. As pointed out in Section \ref{sec-intro}, the identities \eqref{s-39}--\eqref{s-96} imply that $f_{A,B,C,v+L}(q)$ is modular for the given choices of $(A,B,C,v+L)$ in Theorem \ref{thm-rank1}.

We call $(\alpha_{n}(a;q),\beta_{n}(a;q))$ a \emph{Bailey pair relative to $a$} if for all $n\geq 0$,
    \begin{align}
        \beta_{n}(a;q)=\sum_{r=0}^{n}\frac{\alpha_{r}(a;q)}{(q;q)_{n-r}(aq;q)_{n+r}}.
    \end{align}
As a consequence of Bailey's lemma (see e.g. \cite[p.\ 4]{MSZ2008}), we have the following useful transformation formula (see e.g. \cite[Eq.\ (1.2.8)]{MSZ2008}).
    \begin{lemma}\label{lem-BP-id}
        If $(\alpha_{n}(a;q),\beta_{n}(a;q))$ is a Bailey pair, then we have
        \begin{align}
            \sum_{n=0}^{\infty}a^{n}q^{n^2}\beta_{n}(a;q)=\frac{1}{(aq;q)_{\infty}}\sum_{n=0}^{\infty}a^{n}q^{n^2}\alpha_{n}(a;q).
        \end{align}
    \end{lemma}

In order to make our formulas more concise, sometimes we use the notation:
\begin{align}
    J_{m}:=(q^m;q^m)_{\infty}, \quad J_{a,m}:=(q^{a},q^{m-a},q^{m};q^{m})_{\infty}.
\end{align}

Let $q=e^{2\pi i\tau}$ where  $\tau \in \mathbb{H}:=\{\tau\in \mathbb{C}: \mathrm{Im} ~ \tau>0\}$. Recall the Dedekind eta function
\begin{align}\label{eq-eta-defn}
\eta(\tau):=q^{1/24}\prod\limits_{n=1}^\infty (1-q^n)
\end{align}
and the generalized Dedekind eta function
\begin{align}\label{eq-general-eta}
\eta_{m;a}(\tau):=q^{\frac{m}{2}P_2(\frac{a}{m})}\prod\limits_{\begin{smallmatrix}n\equiv \pm a \!\! \pmod{m} \\ n> 0\end{smallmatrix}} (1-q^n).
\end{align}
Here $P_2(t)=\{t\}^2-\{t\}+\frac{1}{6}$ is the second periodic Bernoulli polynomial, $\{t\}=t-\lfloor t\rfloor $ is the fractional part of $t$, $a,m\in \mathbb{Z}$ and $0<a<m$. It is well known that $\eta(\tau)$ is a modular form of weight $1/2$ and $\eta_{m;a}(\tau)$ is a modular function of weight $0$. As a consequence, both $q^{m/24}J_m$ and $q^{m/24+m/2P_2(a/m)} J_{a,m}$ are modular forms of weight $1/2$. Based on this fact, once we express a partial Nahm sum $f_{A,B,0,v+L}(q)$ as a single infinite product expressed by $J_m$ and $J_{a,m}$, it is easy to find the unique scalar $C$ so that $f_{A,B,C,v+L}(q)=q^Cf_{A,B,0,v+L}(q)$ is modular.

Let $v\in \mathbb{Z}^2$ and $L$ be one of the lattices in \eqref{eq-lattice}. Note that $(n_1,n_2)^\mathrm{T}\in v+L$ means that $(n_1,n_2)$ can be written as one of the forms
\begin{equation}
\begin{split}
   & (2i,j), ~~ (2i+1,j), ~~ (i,2j), ~~(i,2j+1), \\
   & (2i,2j),~~ (2i+1,2j),~~ (2i,2j+1), ~~(2i+1,2j+1)
\end{split}
\end{equation}
where $i,j\in \mathbb{Z}$. A key step to calculate rank two partial Nahm sums on lattice cosets $v+L$ is the following  fact.
\begin{lemma}\label{lem-mod2}
Let
\begin{align}
F(u,v)=F(u,v;q):=\sum_{i,j\geq 0} \frac{q^{\frac{1}{2}ai^2+bij+\frac{1}{2}cj^2}u^iv^j}{(q;q)_i(q;q)_j}.
\end{align}
We have
\begin{align}
&\sum_{i,j\geq 0} \frac{q^{2ai^2+2bij+\frac{1}{2}cj^2+2di+ej}}{(q;q)_{2i}(q;q)_j}=\frac{1}{2}\big(F(q^d,q^e)+F(-q^d,q^e)\big), \label{F-add} \\
&\sum_{i,j\geq 0} \frac{q^{2ai^2+2bij+\frac{1}{2}cj^2+(2a+2d)i+(b+e)j}}{(q;q)_{2i+1}(q;q)_j}=\frac{1}{2}q^{-\frac{1}{2}a-d}\big(F(q^d,q^e)-F(-q^d,q^e)\big), \label{F-subtract} \\
  & \sum_{i,j\geq 0} \frac{q^{2ai^2+4bij+2cj^2+2di+2ej}}{(q;q)_{2i}(q;q)_{2j}} \nonumber \\
  &=\frac{1}{4}\big(F(q^d,q^e)+F(q^d,-q^e)+F(-q^d,q^e)+F(-q^d,-q^e)\big),  \label{F-00}\\
  & \sum_{i,j\geq 0} \frac{q^{2ai^2+4bij+2cj^2+2(b+d)i+2(c+e)j}}{(q;q)_{2i}(q;q)_{2j+1}} \nonumber \\
  &=\frac{1}{4}q^{-\frac{1}{2}c-e}\big(F(q^d,q^e)-F(q^d,-q^e)+F(-q^d,q^e)-F(-q^d,-q^e)\big), \label{F-01}\\
  & \sum_{i,j\geq 0} \frac{q^{2ai^2+4bij+2cj^2+2(a+d)i+2(b+e)j}}{(q;q)_{2i+1}(q;q)_{2j}} \nonumber \\
  &=\frac{1}{4}q^{-\frac{1}{2}a-d}\big(F(q^d,q^e)+F(q^d,-q^e)-F(-q^d,q^e)-F(-q^d,-q^e)\big), \label{F-10} \\
  & \sum_{i,j\geq 0} \frac{q^{2ai^2+4bij+2cj^2+2(a+b+d)i+2(b+c+e)j}}{(q;q)_{2i+1}(q;q)_{2j+1}} \nonumber \\
   &=\frac{1}{4}q^{-\frac{1}{2}a-\frac{1}{2}c-b-d-e}\big(F(q^d,q^e)-F(q^d,-q^e)-F(-q^d,q^e)+F(-q^d,-q^e)).\label{F-11}
\end{align}
\end{lemma}
\begin{proof}
The assertions follow immediately from the evaluations of the following sums:
\begin{align*}
&\frac{1}{2}\sum_{i,j\geq 0} \frac{q^{\frac{1}{2}ai^2+bij+\frac{1}{2}cj^2+di+ej}(1 \pm (-1)^i)}{(q;q)_{i}(q;q)_{j}}, \\
&\frac{1}{4}\sum_{i,j\geq 0} \frac{q^{\frac{1}{2}ai^2+bij+\frac{1}{2}cj^2+di+ej}(1+\epsilon (-1)^i)(1 \pm (-1)^j)}{(q;q)_{i}(q;q)_{j}}, \quad \epsilon \in \{1,-1\}. \qedhere
\end{align*}
\end{proof}

\begin{lemma}\label{lem-convergence}
Suppose the function $Q(n_1,n_2,\dots,n_r)$ satisfies
\begin{align*}
Q(n_1,n_2,\dots,n_r)\geq \sum_{i=k+1}^r c_i n_i^{e_i}+M, \quad \forall n_1,n_2,\dots,n_r\geq 0,
\end{align*}
where $k\geq 0$, $M$ is a constant, $e_i\geq 1$ and $c_i$ are some positive constants. Let $u_i \in \mathbb{C}$ satisfy the conditions that $\max(|u_1|,\dots,|u_k|)<1$ and $|u_iq^{c_i}|<1$ whenever $e_i=1$ ($k+1\leq i\leq r$). Then for any subset $S\subseteq \mathbb{Z}^r$,
\begin{align}\label{eq-multi}
    \sum_{(n_1,n_2,\dots,n_r)\in S} \frac{q^{Q(n_1,n_2,\dots,n_r)}u_1^{n_1}u_2^{n_2}\cdots u_r^{n_r}}{(q;q)_{n_1}(q;q)_{n_2}\cdots (q;q)_{n_r}}
\end{align}
converge absolutely when $|q|<1$.
\end{lemma}
\begin{proof}
It suffices to prove the case $S=\mathbb{Z}^r$.
Note that for $|a|,|q|<1$ we have
\begin{align}
&|(a;q)_n|=|(1-a)(1-aq)\cdots (1-aq^{n-1})| \nonumber \\
&\geq (1-|a|)(1-|a||q|)\cdots (1-|a||q|^{n-1})=(|a|;|q|)_n.
\end{align}
For $e=1$, $c>0$ and $u$ satisfying $|uq^c|<1$, using \eqref{euler-1} we have
\begin{align}
&\sum_{n=0}^\infty \left| \frac{u^nq^{cn^e}}{(q;q)_n}\right| \leq  \sum_{n=0}^\infty \frac{|u|^n|q|^{cn}}{(|q|;|q|)_n}=\frac{1}{(|u||q|^c;|q|)_\infty}.
\end{align}
For $e> 1$, $c>0$ and any $u\in \mathbb{C}$ we can find a positive number $\lambda$ such that $|u|<|q|^{-\lambda}$ and an integer $n_0>0$ such that
\begin{align}
    cn^e\geq \lambda n, \quad \forall n\geq n_0.
\end{align}
In this case,
\begin{align}\label{convergence-2}
  &\sum_{n=0}^\infty \left| \frac{u^nq^{cn^e}}{(q;q)_n}\right| = \sum_{n=0}^{n_0-1} \left| \frac{u^nq^{cn^e}}{(q;q)_n}\right|+ \sum_{n=n_0}^\infty \left| \frac{u^nq^{cn^e}}{(q;q)_n}\right|\leq  \sum_{n=0}^{n_0-1}  \frac{|u|^n|q|^{cn^e}}{(|q|;|q|)_n}+\sum_{n=n_0}^\infty \frac{|u|^n|q|^{\lambda n}}{(|q|;|q|)_n} \nonumber \\
  &\leq  \sum_{n=0}^{n_0-1}  \frac{|u|^n|q|^{cn^e}}{(|q|;|q|)_n}+\frac{1}{(|u||q|^\lambda;|q|)_\infty}.
\end{align}
Therefore, we have
\begin{align}\label{proof-convergence}
    & \sum_{(n_1,n_2,\dots,n_r)\in \mathbb{Z}^r}  \left|\frac{q^{Q(n_1,n_2,\dots,n_r)}u_1^{n_1}u_2^{n_2}\cdots u_r^{n_r}}{(q;q)_{n_1}(q;q)_{n_2}\cdots (q;q)_{n_r}}\right| \nonumber \\
    &\leq |q|^M \sum_{(n_1,n_2,\dots,n_r)\in \mathbb{Z}_{\geq 0} ^r}  \frac{|u_1|^{n_1}|u_2|^{n_2}\cdots |u_r|^{n_r}|q|^{c_{k+1}n_{k+1}^{e_{k+1}}+\cdots +c_rn_r^{e_r}}}{(|q|;|q|)_{n_1}(|q|;|q|)_{n_2}\cdots (|q|;|q|)_{n_r}} \nonumber \\
   & =|q|^M \prod\limits_{i=1}^k \sum_{n_i=0}^\infty \frac{|u_i|^{n_i}}{(|q|;|q|)_{n_i}} \times \prod\limits_{j=k+1}^{r} \sum_{n_j=0}^\infty \frac{|u_j|^{n_j}|q|^{c_{j}{n_j}^{e_j}}}{(|q|;|q|)_{n_j}} \nonumber \\
    &=|q|^M\prod\limits_{i=1}^k \frac{1}{(|u_i|;|q|)_\infty} \times \prod\limits_{j=k+1}^r \sum_{n_j=0}^\infty \frac{|u_j|^{n_j}|q|^{c_jn_j^{e_j}}}{(|q|;|q|)_{n_j}}<\infty.
\end{align}
Here the last inequality follows from  \eqref{convergence-2} which indicates that each infinite sum in the second product is convergent. From \eqref{proof-convergence} we get the desired assertion.
\end{proof}

Under the assumption of this lemma, we can interchange the order of summation on the indices $n_1,n_2,\dots,n_r$ in Nahm sums and the series will not change. This fact will be used without mention in our proofs throughout. In particular, as a natural application, we prove the following result which is well-known but a proof is not easy to find in the literature.
\begin{corollary}\label{cor-positive}
Suppose that $A\in \mathbb{Q}^{r\times r}$ is positive definite, $B\in \mathbb{Q}^r$ and $C\in \mathbb{Q}$, then the Nahm sum $f_{A,B,C}(q)$ is absolutely convergent.
\end{corollary}
\begin{proof}
Let $0<\lambda_1\leq \lambda_2\leq \cdots \leq \lambda_r$ be the eigenvalues of $A$. It is well known that
\begin{align}
\frac{1}{2}n^\mathrm{T}An\geq \frac{1}{2}\lambda_1 (n_1^2+n_2^2+\cdots +n_r^2).
\end{align}
Therefore, given any rational vector $B$ and scalar $C$, there exists some constant $M$ such that for all $n_1,n_2,\dots,n_r\geq 0$,
\begin{align}
\frac{1}{2}\lambda_1 (n_1^2+n_2^2+\cdots +n_r^2)+n^\mathrm{T}B+C\geq \frac{1}{4}\lambda_1 (n_1^2+n_2^2+\cdots +n_r^2)+M.
\end{align}
The assertion then follows from Lemma \ref{lem-convergence}.
\end{proof}

\section{Nahm sums for matrix with negative determinant}\label{sec-negative}

\subsection{The cases $(\alpha/2,1,0)$ and $(0,1,\alpha/2)$ with $\alpha>0$} We find that the Nahm sum $f_{A,B,C,v+L}(q)$ is modular for $A=M(\alpha/2,1,0)$ and the choices of $(B,C,v,L)$ in Table \ref{tab-1}.

\begin{table}[htbp]
\caption{Modular quadruples $(A,B,C,v+L)$ with $A=M(\alpha/2,1,0)$ ($\alpha>0$)} \label{tab-1}
\centering
\renewcommand{\arraystretch}{1.5}
\begin{tabular}{|c|c|c|c|c|}
  \hline
  $A$ & $B$ & $C$ & $v$ & $L$ \\ \hline
  \multirow{2}{*}{  $\begin{pmatrix} \alpha/2 & 1 \\ 1 & 0 \end{pmatrix}$} &  $(\alpha/2,1)^\mathrm{T}$ & \multirow{2}{*}{$(6\alpha-1)/24$}  & $(0,0)^\mathrm{T}$ & \multirow{6}{*}{$\mathbb{Z}(2,0)^\mathrm{T}+\mathbb{Z}(0,1)^\mathrm{T}$} \\ \cline{2-2} \cline{4-4}
  & $(0,1)^\mathrm{T}$  &  & $(1,0)^\mathrm{T}$ & \\ \cline{1-4}
  $\begin{pmatrix}
      1/4 & 1 \\ 1 & 0
  \end{pmatrix}$ & $(-1/4,2)^\mathrm{T}$ & $-1/24$ & $(0,0)^\mathrm{T}$ &   \\ \cline{1-4}
  \multirow{2}{*}{ $\begin{pmatrix} 1/2 & 1 \\ 1 & 0 \end{pmatrix}$} &  \multirow{2}{*}{$(0,2)^\mathrm{T}$} & $-1/24$  & $(0,0)^\mathrm{T}$ & \\ \cline{3-4}
  &  & $-7/24$ & $(1,0)^\mathrm{T}$ & \\  \hline
   \multirow{2}{*}{  $\begin{pmatrix} 1 & 1 \\ 1 & 0 \end{pmatrix}$} &  \multirow{2}{*}{ $(1/2,1)^\mathrm{T}$} & \multirow{2}{*}{$1/12$}  & $(0,0)^\mathrm{T}$ & \multirow{2}{*}{$\mathbb{Z}(1,0)^\mathrm{T}+\mathbb{Z}(0,2)^\mathrm{T}$} \\  \cline{4-4}
  & &  & $(0,1)^\mathrm{T}$  &   \\ \hline
\end{tabular}
\end{table}

\begin{theorem}\label{thm-alpha10}
For any $\alpha>0$ we have
    \begin{align}
        \sum_{i,j\geq 0}\frac{q^{\alpha i^2+2ij+\alpha i+j}}{(q;q)_{2i}(q;q)_{j}}&=\frac{(-q^{2\alpha};q^{2\alpha})_\infty^2(q^{2\alpha};q^{2\alpha})_\infty}{(q;q)_{\infty}},  \label{ep2-1}\\
        \sum_{i,j\geq 0} \frac{q^{\alpha i^2+2ij+\alpha i+2j}}{(q;q)_{2i+1}(q;q)_{j}}&=\frac{(-q^{2\alpha};q^{2\alpha})_\infty^2(q^{2\alpha};q^{2\alpha})_\infty}{(q;q)_{\infty}},  \label{ep2-2} \\
         \sum_{i,j\geq 0}\frac{q^{\frac{1}{2}i^2+2ij-\frac{1}{2}i+2j}}{(q;q)_{2i}(q;q)_{j}}&=\frac{2}{(q,q)_\infty},\label{ep2-3}\\
        \sum_{i,j\geq 0} \frac{q^{i^2+2ij+2j}}{(q;q)_{2i}(q;q)_{j}}&=\frac{1}{(q;q)_\infty},\label{ep2-5}\\
        \sum_{i,j\geq 0} \frac{q^{i^2+2ij+i+3j}}{(q;q)_{2i+1}(q;q)_{j}}&=\frac{1}{(q,q)_\infty}, \label{ep2-6} \\
           \sum_{i,j\geq 0} \frac{q^{\frac{1}{2}i^2+2ij+\frac{1}{2}i+2j}}{(q;q)_{i}(q;q)_{2j}}&=\frac{(-q^{6},-q^{10},q^{16};q^{16})_{\infty}}{(q;q)_{\infty}},\label{ep1-1}\\
        \sum_{i,j\geq 0} \frac{q^{\frac{1}{2}i^2+2ij+\frac{3}{2}i+2j}}{(q;q)_{i}(q;q)_{2j+1}}&=\frac{(-q^{2},-q^{14},q^{16};q^{16})_{\infty}}{(q;q)_{\infty}}.\label{ep1-2}
    \end{align}
\end{theorem}
Before giving a proof, we first prove the following identities which are probably known but we cannot find a reference.
\begin{lemma}
We have
\begin{align}
\sum_{n=0}^\infty \frac{q^n}{(q;q)_{2n}}=\frac{(-q^3,-q^5,q^8;q^8)_\infty}{(q;q)_\infty}, \label{add-RR-1} \\
\sum_{n=0}^\infty \frac{q^n}{(q;q)_{2n+1}}=\frac{(-q,-q^7,q^8;q^8)_\infty}{(q;q)_\infty}. \label{add-RR-2}
\end{align}
\end{lemma}
\begin{proof}
For $t\in \{0,1\}$, by \eqref{euler-1} we have
\begin{align}
&\sum_{n=0}^\infty \frac{q^{2n}}{(q^2;q^2)_{2n+t}}=\frac{1}{2}q^{-t}\Big(\sum_{j=0}^{\infty}\frac{q^{j}}{(q^2;q^2)_{j}}+(-1)^t\sum_{j=0}^{\infty}\frac{(-1)^{j}q^{j}}{(q^2;q^2)_{j}}  \Big) \nonumber \\
&=\frac{1}{2}q^{-t}\Big(\frac{1}{(q;q^2)_\infty}+\frac{(-1)^t}{(-q;q^2)_\infty}\Big).
\end{align}
Substituting \eqref{ss-1} and \eqref{ss-2} into it, we obtain \eqref{add-RR-1} and \eqref{add-RR-2}, respectively.
\end{proof}
\begin{proof}[Proof of Theorem \ref{thm-alpha10}]
Let $t\in \{0,1\}$. Summing over $j$ first using \eqref{euler-1}, we have
    \begin{align*}
        \sum_{i,j\geq 0} \frac{q^{\alpha i^2+2ij+\alpha i+(t+1)j}}{(q;q)_{2i+t}(q;q)_{j}}=\sum_{i=0}^{\infty}\frac{q^{\alpha i^2+\alpha i}}{(q;q)_{2i+t}} \times \frac{1}{(q^{2i+t+1};q)_{\infty}}=\frac{1}{(q;q)_{\infty}}\sum_{i=0}^{\infty}q^{\alpha i^2+\alpha i}.
    \end{align*}
Using \eqref{Jacobi}, we obtain \eqref{ep2-1} and \eqref{ep2-2}.

Summing over $j$ first using \eqref{euler-1}, we have
    \begin{align*}
        &\sum_{i,j\geq 0} \frac{q^{\frac{1}{2}i^2+2ij-\frac{1}{2}i+2j}}{(q;q)_{2i}(q;q)_{j}}
        =\sum_{i=0}^{\infty}\frac{q^{\frac{1}{2}i^2-\frac{1}{2}i}}{(q;q)_{2i}(q^{2i+2};q)_\infty}=\frac{1}{(q;q)_\infty}\sum_{i=0}^{\infty}q^{\frac{1}{2}i^2-\frac{1}{2}i}(1-q^{2i+1}) \nonumber \\
        &=\frac{1}{(q;q)_\infty}\sum_{i=0}^{\infty}(q^{\frac{1}{2}i^2-\frac{1}{2}i}-q^{\frac{1}{2}(i+2)^2-\frac{1}{2}(i+2)})=\frac{2}{(q;q)_\infty}, \\
        &\sum_{i,j\ge0}\frac{q^{i^2+2ij+ti+(t+2)j}}{(q;q)_{2i+t}(q;q)_{j}}=\sum_{i=0}^{\infty}\frac{q^{i^2+ti}}{(q;q)_{2i+t}} \times \frac{1}{(q^{2i+t+2};q)_{\infty}}=\frac{1}{(q;q)_{\infty}}\sum_{i=0}^{\infty}q^{i^2+ti}(1-q^{2i+t+1}) \nonumber \\
        &=\frac{1}{(q;q)_\infty}\sum_{i=0}^{\infty}(q^{i^2+ti}-q^{(i+1)^2+t(i+1)})=\frac{1}{(q;q)_\infty}.
    \end{align*}
This proves \eqref{ep2-3}, \eqref{ep2-5} and \eqref{ep2-6}.

Summing over $i$ first using \eqref{euler-2}, we have
    \begin{align*}
        &\sum_{i,j\geq 0}\frac{q^{\frac{1}{2}i^2+2ij+(\frac{1}{2}+t)i+2j}}{(q;q)_{i}(q;q)_{2j+t}}=\sum_{j=0}^{\infty}\frac{q^{2j}(-q^{2j+t+1},q)_{\infty}}{(q;q)_{2j+t}}=(-q,q)_{\infty} \sum_{j=0}^{\infty}\frac{q^{2j}}{(q^2;q^2)_{2j+t}}.
    \end{align*}
Using \eqref{add-RR-1} and \eqref{add-RR-2} we obtain \eqref{ep1-1} and \eqref{ep1-2}, respectively.
\end{proof}
We may also prove some analogous identities with non-modular product sides. For instance, by similar arguments we can prove that
\begin{align}
    \sum_{i,j\geq 0} \frac{q^{i^2+4ij+6j}}{(q^2;q^2)_{2i+1}(q^2;q^2)_{j}}&=\frac{1+q}{(q^2,q^2)_\infty}. \label{ep2-4}
\end{align}

\subsection{The case $(0,1/2,0)$}
We find that the Nahm sum $f_{A,B,C,v+L}(q)$ is modular for $A=M(0, 1/2,0)$ and $(B,C,v,L)$ given in Table \ref{tab-2}.

\begin{table}[H]
\caption{Modular quadruples $(A,B,C,v+L)$ with $A=M(0,1/2,0)$} \label{tab-2}
\centering
\renewcommand{\arraystretch}{1.5}
\begin{tabular}{|c|c|c|c|}
  \hline
  $B$ & $C$ & $v$ & $L$ \\ \hline
   $(1/2,1/2)^\mathrm{T}$ & $1/20$  & $(0,0)^\mathrm{T}$ & \multirow{3}{*}{$\mathbb{Z}(2,0)^\mathrm{T}+\mathbb{Z}(0,2)^\mathrm{T}$} \\ \cline{1-3}
   $(0,1)^\mathrm{T}$  &  $-21/20$ & $(1,0)^\mathrm{T}$ & \\ \cline{1-3}
   $(1,0)^\mathrm{T}$  &  $-21/20$ & $(0,1)^\mathrm{T}$ & \\
   \hline
\end{tabular}
\end{table}
We discovered the following identities to justify their modularity.
\begin{conj}\label{conj-open}
We have
    \begin{align}
        \sum_{i,j\ge0}\frac{q^{2ij+i+j}}{(q;q)_{2i}(q;q)_{2j}}&=\frac{1}{(q;q^2)_\infty^2(q^{2},q^{8};q^{10})_\infty},\label{eq1-1}\\
        \sum_{i,j\ge0}\frac{q^{2ij+i+3j}}{(q;q)_{2i+1}(q;q)_{2j}}&=\frac{1}{(q;q^2)_\infty^2(q^{4},q^{6};q^{10})_\infty}. \label{eq1-2}
    \end{align}
\end{conj}
Though we cannot prove it at this stage, we can reduce the double sums to single sums. For $t=0,1$ we define
    \begin{align*}
        &S_t(x)=S_t(x;q):=\sum_{i,j\ge0}\frac{x^{j}q^{ij+i+tj}}{(q;q)_{2i+t}(q;q)_{j}}=\sum_{i=0}^{\infty}\frac{q^i}{(q;q)_{2i+t}}\sum_{j=0}^{\infty}\frac{x^{j}q^{(i+t)j}}{(q;q)_{j}}\\
        &=\sum_{i=0}^{\infty}\frac{q^i}{(q;q)_{2i+t}(xq^{i+t};q)_{\infty}}=\frac{1}{(x;q)_\infty}\sum_{i=0}^{\infty}\frac{q^i(x;q)_{i+t}}{(q;q)_{2i+t}}.
    \end{align*}
We have
    \begin{align}
        &\sum_{i,j\ge0}\frac{q^{2ij+i+j}}{(q;q)_{2i}(q;q)_{2j}}       =\frac{1}{2}\big(S_0(q^{\frac{1}{2}})+S_0(-q^{\frac{1}{2}})\big) \nonumber \\
        &=\frac{1}{2}\left(\frac{1}{(q^{\frac{1}{2}};q)_\infty}\sum_{i=0}^{\infty}\frac{q^i(q^{\frac{1}{2}};q)_{i}}{(q;q)_{2i}}+\frac{1}{(-q^{\frac{1}{2}};q)_\infty}\sum_{i=0}^{\infty}\frac{q^i(-q^{\frac{1}{2}};q)_{i}}{(q;q)_{2i}}\right) \label{id-single-1}
    \end{align}
and
    \begin{align}
        &\sum_{i,j\ge0}\frac{q^{2ij+i+3j}}{(q;q)_{2i+1}(q;q)_{2j}}      =\frac{1}{2}\big(S_1(q^{\frac{1}{2}})+S_1(-q^{\frac{1}{2}})\big) \nonumber \\
        &=\frac{1}{2}\left(\frac{1}{(q^{\frac{1}{2}};q)_\infty}\sum_{i=0}^{\infty}\frac{q^i(q^{\frac{1}{2}};q)_{i+1}}{(q;q)_{2i+1}}+\frac{1}{(-q^{\frac{1}{2}};q)_\infty}\sum_{i=0}^{\infty}\frac{q^i(-q^{\frac{1}{2}};q)_{i+1}}{(q;q)_{2i+1}}\right). \label{id-single-2}
    \end{align}
It is not clear whether \eqref{id-single-1} and \eqref{id-single-2} are helpful or not to prove Conjecture \ref{conj-open}.

\subsection{The cases $(1/2,1/2,0)$ and $(0,1/2,1/2)$}
We find that the Nahm sum $f_{A,B,C,v+L}(q)$ is modular for $A=M(1/2, 1/2,0)$ and the choices of $(B,C,v,L)$ in Table \ref{tab-3}.

\begin{table}[htbp]
\caption{Modular quadruples $(A,B,C,v+L)$ with $A=M(1/2,1/2,0)$} \label{tab-3}
\centering
\renewcommand{\arraystretch}{1.5}
\begin{tabular}{|c|c|c|c|}
  \hline
   $B$ & $C$ & $v$ & $L$ \\ \hline
   $(0,1/2)^\mathrm{T}$ & $-1/84$  & $(0,0)^\mathrm{T}$ & \multirow{3}{*}{$\mathbb{Z}(2,0)^\mathrm{T}+\mathbb{Z}(0,1)^\mathrm{T}$} \\ \cline{1-3}
   $(1/2,1/2)^\mathrm{T}$  & $5/84$ & $(0,0)^\mathrm{T}$ & \\ \cline{1-3}
   $(0,1)^\mathrm{T}$ & $-1/21$ &  $(1,0)^\mathrm{T}$ & \\
  \hline
   $(0,1/2)^\mathrm{T}$ & $-1/84$ & $(0,0)^\mathrm{T}$, $(0,1)^\mathrm{T}$ & \multirow{3}{*}{$\mathbb{Z}(2,0)^\mathrm{T}+\mathbb{Z}(0,2)^\mathrm{T}$} \\ \cline{1-3}
    $(1/2,1/2)^\mathrm{T}$ &  $5/84$ & $(0,0)^\mathrm{T}$, $(0,1)^\mathrm{T}$  &  \\
    \cline{1-3}
    $(0,1)^\mathrm{T}$ & $-1/21$ & $(1,0)^\mathrm{T}$, $(1,1)^\mathrm{T}$ &  \\
    \hline
\end{tabular}
\end{table}
To justify the modularity, we establish two theorems.
\begin{theorem}\label{thm-12120-1}
    We have
    \begin{align}
        \sum_{i,j\geq 0} \frac{q^{2i^2+2ij+j}}{(q^2;q^2)_{2i}(q^2;q^2)_{j}}&=\frac{(q^3,q^4,q^7;q^7)_\infty}{(q;q)_\infty},\label{ep4-1}\\
        \sum_{i,j\geq 0} \frac{q^{2i^2+2ij+2i+j}}{(q^2;q^2)_{2i}(q^2;q^2)_{j}}&=\frac{(q^2,q^5,q^7;q^7)_\infty}{(q;q)_\infty},\label{ep4-2}\\
        \sum_{i,j\geq 0} \frac{q^{2i^2+2ij+2i+3j}}{(q^2;q^2)_{2i+1}(q^2;q^2)_{j}}&=\frac{(q,q^6,q^7;q^7)_\infty}{(q;q)_\infty}.\label{ep4-3}
    \end{align}
\end{theorem}
\begin{proof}
Summing over $j$ first using \eqref{euler-1}, we have
\begin{align*}
         \sum_{i,j\geq 0} \frac{q^{2i^2+2ij+j}}{(q^2;q^2)_{2i}(q^2;q^2)_{j}}=\sum_{i=0}^{\infty}\frac{q^{2i^2}}{(q^2;q^2)_{2i}(q^{2i+1};q^2)_\infty}=\frac{1}{(q;q^2)_{\infty}}\sum_{i=0}^{\infty}\frac{q^{2i^2}(q;q^2)_{i}}{(q^2;q^2)_{2i}}.
    \end{align*}
Now using \eqref{s-33} we obtain \eqref{ep4-1}.
    The other two identities \eqref{ep4-2} and \eqref{ep4-3} can be proved similarly using \eqref{s-32} and \eqref{s-31}. We omit the details.
\end{proof}

The following result gives refinements of Theorem \ref{thm-12120-1} and justifies the modularity of the modular quadruples in Table \ref{tab-3} for the lattice $L=\mathbb{Z}(2,0)^\mathrm{T}+\mathbb{Z}(0,2)^\mathrm{T}$.

\begin{theorem}
    We have
    \begin{align}
        \sum_{i,j\ge0}\frac{q^{i^2+2ij+j}}{(q;q)_{2i}(q;q)_{2j}}&=\frac{1}{(q;q^2)_\infty^2(q^{4},q^{8},q^{20},q^{24};q^{28})_\infty},\label{eq2-1}\\
        \sum_{i,j\ge0}\frac{q^{i^2+2ij+i+j}}{(q;q)_{2i}(q;q)_{2j+1}}&=\frac{(q^{8},q^{20},q^{14},q^{28};q^{28})_\infty}{(q;q)_\infty(q,q^{7},q^{10},q^{13},q^{15},q^{18},q^{21},q^{27};q^{28})_\infty},\label{eq2-3}\\
        \sum_{i,j\ge0}\frac{q^{i^2+2ij+i+j}}{(q;q)_{2i}(q;q)_{2j}}&=\frac{(q^{4},q^{24},q^{14},q^{28};q^{28})_\infty}{(q;q)_\infty(q^{2},q^{3},q^{7},q^{11},q^{17},q^{21},q^{25},q^{26};q^{28})_\infty},\label{eq2-2}\\
        \sum_{i,j\ge0}\frac{q^{i^2+2ij+2i+j}}{(q;q)_{2i}(q;q)_{2j+1}}&=\frac{1}{(q;q^2)_\infty^2(q^{4},q^{12},q^{16},q^{24};q^{28})_\infty},\label{eq2-4}\\
        \sum_{i,j\ge0}\frac{q^{i^2+2ij+i+3j}}{(q;q)_{2i+1}(q;q)_{2j}}&=\frac{(q^{12},q^{16},q^{14},q^{28};q^{28})_\infty}{(q;q)_\infty(q^{5},q^{6},q^{7},q^{9},q^{19},q^{21},q^{22},q^{23};q^{28})_\infty},\label{eq2-5}\\
        \sum_{i,j\ge0}\frac{q^{i^2+2ij+2i+3j}}{(q;q)_{2i+1}(q;q)_{2j+1}}&=\frac{1}{(q;q^2)_\infty^2(q^{8},q^{12},q^{16},q^{20};q^{28})_\infty}.\label{eq2-6}
    \end{align}
\end{theorem}
\begin{proof}
We have
    \begin{align}
        &\sum_{i,j\ge0}\frac{q^{i^2+2ij+j}}{(q;q)_{2i}(q;q)_{2j}}=\frac{1}{2}\left(\sum_{i,j\ge0}\frac{q^{i^2+ij+\frac{1}{2}j}}{(q;q)_{2i}(q;q)_{j}}+\sum_{i,j\ge0}\frac{(-1)^{j}q^{i^2+ij+\frac{1}{2}j}}{(q;q)_{2i}(q;q)_{j}}\right) \nonumber \\
        &=\frac{1}{2}\left(\frac{(q^{3/2},q^2,q^{7/2};q^{7/2})_\infty}{(q^{1/2};q^{1/2})_\infty}+\frac{(-q^{3/2},q^2,-q^{7/2};-q^{7/2})_\infty}{(-q^{1/2};-q^{1/2})_\infty} \right). \label{add-4.2-1}
    \end{align}
Here the last equality follows from \eqref{ep4-1}. Using the method in \cite{Frye-Garvan}, we can prove \eqref{eq2-1}. The remaining identities can be proved similarly using \eqref{ep4-1}--\eqref{ep4-3}.
\end{proof}

\subsection{The case $(0,1,0)$}
We find only one modular partial Nahm sum for the matrix $A=M(0,1,0)$ and already state it in Theorem \ref{thm-010}.
\begin{proof}[Proof of Theorem \ref{thm-010}]
We define $S(a,z;q)$ as follows:
\begin{align}\label{Saz}
    &S(a,z;q):=\sum_{i=0}^\infty \frac{z^{2i+1}(a;q)_{2i+1}}{(q;q)_{2i+1}}=\frac{1}{2}\sum_{i=0}^\infty\left( \frac{(a;q)_nz^n}{(q;q)_n}-\frac{(a;q)_n(-z)^n}{(q;q)_n}\right)\notag\\
    &=\frac{1}{2}\left(\frac{(az;q)_\infty}{(z;q)_\infty}-\frac{(-az;q)_\infty}{(-z;q)_\infty}\right).\quad\text{(by \eqref{q-binomial})}
\end{align}
We have
    \begin{align*}
        &\sum_{i,j\ge0}\frac{q^{4ij+i+3j}}{(q;q)_{2i+1}(q;q)_{2j}}=\frac{1}{2}\left(\sum_{i,j\ge0}\frac{q^{2ij+i+\frac{3}{2}j}}{(q;q)_{2i+1}(q;q)_{j}}+\sum_{i,j\ge0}\frac{(-1)^{j}q^{2ij+i+\frac{3}{2}j}}{(q;q)_{2i+1}(q;q)_{j}}\right)\\
        &=\frac{1}{2}\left(\sum_{i=0}^{\infty}\frac{q^{i}}{(q;q)_{2i+1}}\sum_{j=0}^{\infty}\frac{q^{2ij+\frac{3}{2}j}}{(q;q)_{j}}+\sum_{i=0}^{\infty}\frac{q^{i}}{(q;q)_{2i+1}}\sum_{j=0}^{\infty}\frac{(-1)^{j}q^{2ij+\frac{3}{2}j}}{(q;q)_{j}} \right)\\
        &=\frac{1}{2}\left(\sum_{i=0}^{\infty}\frac{q^{i}}{(q;q)_{2i+1}}\times \frac{1}{(q^{2i+\frac{3}{2}};q)_{\infty}}+\sum_{i=0}^{\infty}\frac{q^{i}}{(q;q)_{2i+1}} \times \frac{1}{(-q^{2i+\frac{3}{2}};q)_{\infty}}\right)\\
        &=\frac{1}{2}\left(\frac{1}{(q^{\frac{1}{2}};q)_{\infty}}\sum_{i=0}^{\infty}\frac{q^{i}(q^{\frac{1}{2}};q)_{2i+1}}{(q;q)_{2i+1}}+\frac{1}{(-q^{\frac{1}{2}};q)_{\infty}}\sum_{i=0}^{\infty}\frac{q^{i}(-q^{\frac{1}{2}};q)_{2i+1}}{(q;q)_{2i+1}}\right)\\
        &=\frac{1}{4}\Big(\frac{q^{-\frac{1}{2}}}{(q^{\frac{1}{2}};q)_{\infty}}\Big(\frac{(q;q)_\infty}{(q^{\frac{1}{2}};q)_\infty}-\frac{(-q;q)_\infty}{(-q^{\frac{1}{2}};q)_\infty}\Big)+\frac{q^{-\frac{1}{2}}}{(-q^{\frac{1}{2}};q)_{\infty}}\Big(\frac{(-q;q)_\infty}{(q^{\frac{1}{2}};q)_\infty}-\frac{(q;q)_\infty}{(-q^{\frac{1}{2}};q)_\infty}\Big)\Big)\\
       &=\frac{q^{-\frac{1}{2}}}{4(q;q^2)_\infty^2}\left((-q^{\frac{1}{2}},-q^{\frac{1}{2}},q;q)_\infty-(q^{\frac{1}{2}},q^{\frac{1}{2}},q;q)_\infty\right).
    \end{align*}
Here for the last second equality we used \eqref{Saz}. Upon simplifications using \eqref{jac odd} with $(a,b)=(1/2,0)$, we obtain \eqref{eq3-2}.
\end{proof}
In the same way, we can prove some other identities with non-modular product side such as
 \begin{align}
        \sum_{i,j\ge0}\frac{q^{4ij+i+j}}{(q;q)_{2i+1}(q;q)_{2j}}&=\frac{1}{1-q}\times\frac{1}{(q;q^2)_\infty^2},\label{eq3-1}\\
        \sum_{i,j\ge0}\frac{q^{4ij+i+j}}{(q;q)_{2i+1}(q;q)_{2j+1}}&=\frac{1+q}{1-q}\times \frac{1}{(q;q^2)_\infty^2},\label{eq3-3}\\
        \sum_{i,j\ge0}\frac{q^{4ij+3i+j}}{(q;q)_{2i+1}(q;q)_{2j+1}}&=\frac{1}{1-q}\times \frac{1}{(q;q^2)_\infty^2}.\label{eq3-4}
    \end{align}

We did not find any modular partial Nahm sums $f_{A,B,C,v+L}(q)$ for $A=M(0,1,0)$ and $L$ being $\mathbb{Z}(2,0)^\mathrm{T}+\mathbb{Z}(0,1)^\mathrm{T}$ or $\mathbb{Z}(1,0)^\mathrm{T}+\mathbb{Z}(0,2)^\mathrm{T}$. Instead, we can prove some sum-product identities and we list some examples.
\begin{theorem}
    For $|u|<1$, we have
    \begin{align}
        \sum_{i,j\ge0}\frac{q^{2ij+j}u^i}{(q;q)_{2i}(q;q)_{j}}&=\frac{1}{(1-u)(q;q)_\infty},\label{add-1}\\
        \sum_{i,j\ge0}\frac{q^{2ij+2j}u^i}{(q;q)_{2i}(q;q)_{j}}&=\frac{1-uq^2-q+uq}{(1-u)(1-uq^2)(q;q)_\infty},\label{add-2}\\
        \sum_{i,j\ge0}\frac{q^{2ij+2j}u^i}{(q;q)_{2i+1}(q;q)_{j}}&=\frac{1}{(1-u)(q;q)_\infty},\label{add-3}\\
        \sum_{i,j\ge0}\frac{q^{2ij+3j}u^i}{(q;q)_{2i+1}(q;q)_{j}}&=\frac{1-q^2}{(1-u)(1-uq^2)(q;q)_\infty}.\label{add-4}
    \end{align}
\end{theorem}
For any $t\in \mathbb{R}$, $u=q^t$ does not make the product-side being modular.
\begin{proof}
Summing over $j$ first using \eqref{euler-1} we have
    \begin{align*}
         \sum_{i,j\ge0}\frac{q^{2ij+j}u^i}{(q;q)_{2i}(q;q)_{j}}=\sum_{i=0}^{\infty}\frac{u^i}{(q;q)_{2i}(q^{2i+1};q)_{\infty}}=\frac{1}{(q;q)_{\infty}}\sum_{i=0}^{\infty}u^i=\frac{1}{(1-u)(q;q)_\infty}.
    \end{align*}
This proves \eqref{add-1}. The other identities can be proved similarly.
\end{proof}

\section{Nahm sums for matrix with zero determinant}\label{sec-zero}

\subsection{The case $(1,1,1)$}\label{subsec-111}
We find that the Nahm sum $f_{A,B,C,v+L}(q)$ is modular for $A=M(1,1,1)$ and the choices of $(B,C,v,L)$ in Table \ref{tab-4}.
\begin{table}[H]
\caption{Modular quadruples $(A,B,C,v+L)$ with $A=M(1,1,1)$} \label{tab-4}
\centering
\renewcommand{\arraystretch}{1.5}
\begin{tabular}{|c|c|c|c|}
  \hline
  $B$ & $C$ & $v$ & $L$ \\ \hline
   $(0,1/2)^\mathrm{T}$ & $-1/120$ & $(0,0)^\mathrm{T}$, $(1,0)^\mathrm{T}$ & \multirow{2}{*}{$\mathbb{Z}(2,0)^\mathrm{T}+\mathbb{Z}(0,1)^\mathrm{T}$} \\  \cline{1-3}
   $(1,1/2)^\mathrm{T}$ & $11/120$ & $(0,0)^\mathrm{T}$, $(1,0)^\mathrm{T}$ &  \\ \hline
   $(1/2,0)^\mathrm{T}$ & $-1/120$ & $(0,0)^\mathrm{T}$, $(0,1)^\mathrm{T}$ &  \multirow{2}{*}{$\mathbb{Z}(1,0)^\mathrm{T}+\mathbb{Z}(0,2)^\mathrm{T}$} \\  \cline{1-3}
   $(1/2,1)^\mathrm{T}$ & $11/120$ & $(0,0)^\mathrm{T}$, $(0,1)^\mathrm{T}$ & \\ \hline
\end{tabular}
\end{table}
Since $\frac{1}{2}n^{\mathrm{T}} A n$ is symmetric in $i$ and $j$, interchanging the two coordinates in $B,v$ in the first two rows generate modular quadruples in the last two rows for free.

The modularity is justified by the following identities.
\begin{theorem}
    We have
    \begin{align}
        \sum_{i,j\geq 0}\frac{q^{2i^2+2ij+\frac{1}{2}j^2+\frac{1}{2}j}}{(q;q)_{2i}(q;q)_{j}}&=\frac{(q^{4},q^{16},q^{20};q^{20})_{\infty}(q^{12},q^{28};q^{40})_{\infty}}{(q;q)_{\infty}},\label{ep8-1}\\
        \sum_{i,j\geq 0} \frac{q^{2i^2+2ij+\frac{1}{2}j^2+2i+\frac{3}{2}j}}{(q;q)_{2i+1}(q;q)_{j}}&=\frac{(q^{6},q^{14},q^{20};q^{20})_{\infty}(q^{8},q^{32};q^{40})_{\infty}}{(q;q)_{\infty}},\label{ep8-3}\\
        \sum_{i,j\geq 0} \frac{q^{2i^2+2ij+\frac{1}{2}j^2+2i+\frac{1}{2}j}}{(q;q)_{2i}(q;q)_{j}}&=\frac{(q^{2},q^{18},q^{20};q^{20})_{\infty}(q^{16},q^{24};q^{40})_{\infty}}{(q;q)_{\infty}},\label{ep8-2}\\
        \sum_{i,j\geq 0} \frac{q^{2i^2+2ij+\frac{1}{2}j^2+4i+\frac{3}{2}j}}{(q;q)_{2i+1}(q;q)_{j}}&=\frac{(q^{8},q^{12},q^{20};q^{20})_{\infty}(q^{4},q^{36};q^{40})_{\infty}}{(q;q)_{\infty}}.\label{ep8-4}
    \end{align}
\end{theorem}
\begin{proof}
Summing over $j$ first using \eqref{euler-2}, we have
\begin{align*}
\sum_{i,j\geq 0} \frac{q^{2i^2+2ij+\frac{1}{2}j^2+\frac{1}{2}j}}{(q;q)_{2i}(q;q)_{j}}=\sum_{i=0}^{\infty}\frac{q^{2i^2}}{(q;q)_{2i}}(-q^{2i+1};q)_\infty=(-q;q)_\infty\sum_{i=0}^{\infty}\frac{q^{2i^2}}{(q^2;q^2)_{2i}}.
\end{align*}
This proves \eqref{ep8-1} upon using \eqref{s-79}. The remaining identities can be proved similarly using   \eqref{s-94}, \eqref{s-99}  and \eqref{s-96}.
\end{proof}

\subsection{The case $(1/2, 1/2, 1/2)$}
We find that the Nahm sum $f_{A,B,C,v+L}(q)$ is modular for $A=M(1/2,1/2,1/2)$ and the choices of $(B,C,v,L)$ in Table \ref{tab-5}. Note that the entries $v_1,v_2$ in Table \ref{tab-5} can take any integer values.
\begin{table}[H]
\caption{Modular quadruples $(A,B,C,v+L)$ with $A=M(1/2,1/2,1/2)$} \label{tab-5}
\centering
\renewcommand{\arraystretch}{1.5}
\begin{tabular}{|c|c|c|c|}
  \hline
  $B$ & $C$ & $v$ & $L$ \\ \hline
   {$(0,0)^\mathrm{T}$} & {$-1/24$} & $(v_1,v_2)^\mathrm{T}$ & \multirow{5}{*}{$\mathbb{Z}(2,0)^\mathrm{T}+\mathbb{Z}(0,2)^\mathrm{T}$} \\ \cline{1-3}
  {$(0,1/2)^\mathrm{T}$} & {$-1/96$} & $(v_1,v_2)^\mathrm{T}$ &  \\  \cline{1-3}
  $(-1/2,-1/2)^\mathrm{T}$ & $-5/24$ &  $(1,0)^\mathrm{T}$,  $(0,1)^\mathrm{T}$ & \\ \cline{1-3}
   $(-1/2,1/2)^\mathrm{T}$ & $5/24$ &  $(1,0)^\mathrm{T}$, $(0,1)^\mathrm{T}$ & \\
   \cline{1-3}
   $(1/2,-1/2)^\mathrm{T}$ & $5/24$ &  $(1,0)^\mathrm{T}$, $(0,1)^\mathrm{T}$ & \\
   \hline
\end{tabular}
\end{table}

\begin{theorem}\label{thm12-triple}
    We have
    \begin{align}
        \sum_{i,j\ge0}\frac{q^{i^2+2ij+j^2}}{(q;q)_{2i}(q;q)_{2j}}&=\frac{(-q;q)_\infty^2(-q^8;q^8)_\infty}{(-q^{2},-q^{14};q^{16})_\infty},\label{eq8-1}\\
        \sum_{i,j\ge0}\frac{q^{i^2+2ij+j^2+i+j}}{(q;q)_{2i+1}(q;q)_{2j}}&=\frac{1}{(q;q^2)_\infty (q^2;q^4)_\infty^2},\label{eq8-5}\\
        \sum_{i,j\ge0}\frac{q^{i^2+2ij+j^2+2i+2j}}{(q;q)_{2i+1}(q;q)_{2j+1}}&=\frac{(-q;q)_\infty^2(-q^8;q^8)_\infty}{(-q^{6},-q^{10};q^{16})_\infty},\label{eq8-19} \\
        \sum_{i,j\ge0}\frac{q^{i^2+2ij+j^2+j}}{(q;q)_{2i}(q;q)_{2j}}&=\frac{(-q^{7},-q^{9},q^{16};q^{16})_\infty}{(q;q)_\infty},\label{eq8-2}\\
        \sum_{i,j\ge0}\frac{q^{i^2+2ij+j^2+i+2j}}{(q;q)_{2i+1}(q;q)_{2j}}&=\frac{(-q^{5},-q^{11},q^{16};q^{16})_\infty}{(q;q)_\infty},\label{eq8-6}\\
        \sum_{i,j\ge0}\frac{q^{i^2+2ij+j^2+i+2j}}{(q;q)_{2i}(q;q)_{2j+1}}&=\frac{(-q^{3},-q^{13},q^{16};q^{16})_\infty}{(q;q)_\infty},\label{eq8-8}\\
         \sum_{i,j\ge0}\frac{q^{i^2+2ij+j^2+2i+3j}}{(q;q)_{2i+1}(q;q)_{2j+1}}&=\frac{(-q,-q^{15},q^{16};q^{16})_\infty}{(q;q)_\infty}, \label{eq8-10} \\
        \sum_{i,j\ge0}\frac{q^{i^2+2ij+j^2}}{(q;q)_{2i+1}(q;q)_{2j}}&=\frac{(q^2;q^4)_\infty^2}{(q;q^2)_\infty^3},\label{eq8-3}\\
        \sum_{i,j\ge0}\frac{q^{i^2+2ij+j^2+2j}}{(q;q)_{2i+1}(q;q)_{2j}}&=\frac{(-q;q)_\infty^2(-q^8;q^8)_\infty}{(-q^{2},-q^{14};q^{16})_\infty},\label{eq8-4}\\
        \sum_{i,j\ge0}\frac{q^{i^2+2ij+j^2+2i}}{(q;q)_{2i+1}(q;q)_{2j}}&=\frac{(-q;q)_\infty^2(-q^8;q^8)_\infty}{(-q^{6},-q^{10};q^{16})_\infty}.\label{eq8-7}
    \end{align}
\end{theorem}

Before giving a proof, we establish the following lemma.
\begin{lemma}
We have
\begin{align}
     \sum_{i=0}^{n}q^{i}{   n \brack i}_{q^2}&=(-q;q)_n, \label{finite-1} \\
    \sum_{i=0}^{n}(-1)^{i}{   n \brack i} &= \begin{cases}
        (q;q)_{2k}/(q^2;q^2)_{k},  & n=2k,k\in \mathbb{N} \\
        0,  & n=2k+1,k\in \mathbb{N}
        \end{cases}, \label{finite-2} \\
        \sum_{i=0}^{n}(-1)^{i}q^{i}{ n \brack i} &= \begin{cases}
        (q;q)_{2k}/(q^2;q^2)_{k},  &n=2k,k\in \mathbb{N} \\
        (q;q)_{2k+1}/(q^2;q^2)_{k},  &n=2k+1,k\in \mathbb{N}
        \end{cases}. \label{finite-3}
\end{align}
\end{lemma}
\begin{proof}
Comparing the coefficients of $z^n$ on both sides of the following identities:
    \begin{align*}
       & \sum_{i=0}^\infty  \frac{z^iq^i}{(q^2;q^2)_i} \sum_{j=0}^\infty \frac{z^j}{(q^2;q^2)_j}=\frac{1}{(zq;q^2)_\infty} \times \frac{1}{(z;q^2)_\infty}=\frac{1}{(z;q)_\infty}=\sum_{n=0}^\infty \frac{z^n}{(q;q)_n}, \\
       & \sum_{i=0}^\infty  \frac{(-z)^i}{(q;q)_i} \sum_{j=0}^\infty \frac{z^j}{(q;q)_j}=\frac{1}{(-z;q)_\infty} \times \frac{1}{(z;q)_\infty}=\frac{1}{(z^2;q^2)_\infty}=\sum_{n=0}^\infty \frac{z^{2n}}{(q^2;q^2)_n}, \\
       & \sum_{i=0}^\infty  \frac{(-z)^iq^i}{(q;q)_i} \sum_{j=0}^\infty \frac{z^j}{(q;q)_j}=\frac{1}{(-zq;q)_\infty} \times \frac{1}{(z;q)_\infty}=\frac{1+z}{(z^2;q^2)_\infty}=\sum_{n=0}^\infty \frac{z^{2n}(1+z)}{(q^2;q^2)_n},
    \end{align*}
we obtain \eqref{finite-1}--\eqref{finite-3}, respectively.
\end{proof}

We define $R(x,y)$ as follows:
    \begin{align}
       R(x,y)=R(x,y;q):=\sum_{i,j\ge0}\frac{q^{\frac{1}{4}i^2+\frac{1}{2}ij+\frac{1}{4}j^2}x^iy^j}{(q;q)_i(q;q)_j}.
    \end{align}
It is easy to see that $R(x,y)=R(y,x)$.
\begin{lemma}\label{lem-R}
We have
\begin{align}
        &R(1,1)=\frac{(-q^{1/4},-q^{1/4},q^{1/2};q^{1/2})_\infty}{(q;q)_\infty},\label{R1++0}\\
       & R(q^{\frac{1}{2}},q^{-\frac{1}{2}})=\frac{(-q^{3/4},-q^{-1/4},q^{1/2};q^{1/2})_\infty}{(q;q)_\infty}\label{R1++2}, \\
       & R(1,q^{\frac{1}{2}})=(-q^{1/4};q^{1/2})_{\infty}, \label{R1++1} \\          &R(-1,-1;q)=\frac{(q^{1/4},q^{1/4},q^{1/2};q^{1/2})_\infty}{(q;q)_\infty},\label{R1--0}\\
        &R(-1,-q^{\frac{1}{2}})=(q^{1/4};q^{1/2})_{\infty},\label{R1--1}\\
       & R(-q^{\frac{1}{2}},-q^{-\frac{1}{2}})=\frac{(q^{3/4},q^{-1/4},q^{1/2};q^{1/2})_\infty}{(q;q)_\infty}\label{R1--2}, \\
       & R(-1,1)=(-q;q^2)_\infty, \label{R1+-0} \\
         &R(-q^{\frac{1}{2}},1)+R(q^{\frac{1}{2}},-1)=2\frac{(q^{3/2},q^{5/2},q^{4};q^{4})_{\infty}}{(q;q)_{\infty}},\label{R1+1}\\
        &R(-q^{\frac{1}{2}},1)-R(q^{\frac{1}{2}},-1)=2\frac{q^{1/4}(q^{1/2},q^{7/2},q^{4};q^{4})_{\infty}}{(q;q)_{\infty}}, \label{R1-1} \\
        &R(-q^{\frac{1}{2}},q^{-\frac{1}{2}})+R(q^{\frac{1}{2}},-q^{-\frac{1}{2}})=4(-q^2;q^2)_\infty, \label{R1-3} \\
        & R(-q^{\frac{1}{2}},q^{-\frac{1}{2}})-R(q^{\frac{1}{2}},-q^{-\frac{1}{2}})=2q^{-1/4}(-q;q^2)_\infty.\label{R1-2}
    \end{align}
\end{lemma}
\begin{proof}
 Recall the following identities found by Cao and Wang \cite[Theorem 3.4]{Cao-Wang}:
    \begin{align}
        \sum_{i,j\ge0}\frac{u^{i-j}q^{\frac{a+1}{2}i^2-aij+\frac{a+1}{2}j^2+\frac{a-1}{2}(i-j)}}{(q;q)_{i}(q;q)_{j}}=\frac{(-uq^{a},-q/u,q^{a+1};q^{a+1})_{\infty}}{(q;q)_{\infty}}. \label{cao-wang 1}
    \end{align}
Setting $a=-1/2$ and $u=q^{3/4},q^{5/4}$ we obtain \eqref{R1++0} and \eqref{R1++2}.

We have
\begin{align*}
         &\sum_{i,j\ge0}\frac{q^{\frac{1}{4}i^2+\frac{1}{2}ij+\frac{1}{4}j^2+\frac{1}{2}j}}{(q;q)_i(q;q)_j}=\sum_{n=0}^{\infty}\frac{q^{\frac{1}{4}n^2}}{(q,q)_n}\sum_{j=0}^{n}q^{\frac{1}{2}j}{ n \brack j}       =\sum_{n=0}^{\infty}\frac{q^{\frac{1}{4}n^2}}{(q^{\frac{1}{2}},q^{\frac{1}{2}})_n}=(-q^{1/4};q^{1/2})_{\infty}.
    \end{align*}
Here for the last second equality we used \eqref{finite-1}. This proves \eqref{R1++1}.

Replacing $q^{\frac{1}{4}}$ by $-q^{\frac{1}{4}}$ in \eqref{R1++0}--\eqref{R1++1}, we obtain \eqref{R1--0}--\eqref{R1--2}.

The identity \eqref{R1+-0} follows from
    \begin{align}
       & R(-1,1)=\sum_{i,j\ge0}\frac{(-1)^{i}q^{\frac{1}{4}i^2+\frac{1}{2}ij+\frac{1}{4}j^2}}{(q;q)_i(q;q)_j}=\sum_{n=0}^{\infty}\frac{q^{\frac{1}{4}n^2}}{(q;q)_n}\sum_{i=0}^{n}(-1)^i{   n \brack i}\nonumber \\
        &=\sum_{n=0}^{\infty}\frac{q^{n^2}}{(q^2;q^2)_n}=(-q;q^2)_\infty. \quad \text{(by \eqref{finite-2} and \eqref{euler-2})}
    \end{align}

Replacing $q$ by $-q$ in \eqref{finite-1}, we have
    \begin{align}
        \sum_{i=0}^{n}(-1)^{i}q^{i}{ n \brack i}_{q^2}=\frac{(q^2;q^2)_{n}}{(-q;-q)_{n}}.\label{finite-1'}
    \end{align}
   Using \eqref{finite-1'} we have
    \begin{align}
        &R(-q^{\frac{1}{2}},1)=\sum_{i,j\ge0}\frac{(-1)^{i}q^{\frac{1}{4}i^2+\frac{1}{2}ij+\frac{1}{4}j^2+\frac{1}{2}i}}{(q;q)_i(q;q)_j}=\sum_{n=0}^{\infty}\frac{q^{\frac{1}{4}n^2}}{(q;q)_n}\sum_{i=0}^{n}(-1)^{i}q^{\frac{1}{2}i}{   n \brack i}\notag\\
        &=\sum_{n=0}^{\infty}\frac{q^{\frac{1}{4}n^2}(q;q)_n}{(q;q)_n(-q^{\frac{1}{2}};-q^{\frac{1}{2}})_{n}}=\sum_{n=0}^{\infty}\frac{q^{\frac{1}{4}n^2}}{(-q^{\frac{1}{2}};-q^{\frac{1}{2}})_{n}},\label{R1+-1}\\
        &R(q^{\frac{1}{2}},-1)=\sum_{i,j\ge0}\frac{(-1)^{j}q^{\frac{1}{4}i^2+\frac{1}{2}ij+\frac{1}{4}j^2+\frac{1}{2}i}}{(q;q)_i(q;q)_j}=\sum_{n=0}^{\infty}\frac{(-1)^{n}q^{\frac{1}{4}n^2}}{(q;q)_n}\sum_{i=0}^{n}(-1)^{i}q^{\frac{1}{2}i}{   n \brack i}\notag\\
        &=\sum_{n=0}^{\infty}\frac{(-1)^{n}q^{\frac{1}{4}n^2}(q;q)_n}{(q;q)_n(-q^{\frac{1}{2}};-q^{\frac{1}{2}})_{n}}=\sum_{n=0}^{\infty}\frac{(-1)^{n}q^{\frac{1}{4}n^2}}{(-q^{\frac{1}{2}};-q^{\frac{1}{2}})_{n}}.\label{R1-+1}
    \end{align}
Using \eqref{s-39} and \eqref{s-38} with $q$ replaced by $-q^{1/2}$, we have
    \begin{align}
        &R(-q^{\frac{1}{2}},1)+R(q^{\frac{1}{2}},-1)=\sum_{n=0}^{\infty}\frac{q^{\frac{1}{4}n^2}}{(-q^{\frac{1}{2}};-q^{\frac{1}{2}})_{n}}+\sum_{n=0}^{\infty}\frac{(-1)^{n}q^{\frac{1}{4}n^2}}{(-q^{\frac{1}{2}};-q^{\frac{1}{2}})_{n}}\notag\\
        &=2\sum_{n=0}^{\infty}\frac{q^{n^2}}{(-q^{\frac{1}{2}};-q^{\frac{1}{2}})_{2n}}=2\frac{(q^{3/2},q^{5/2},q^{4};q^{4})_{\infty}}{(q;q)_{\infty}}, \\
        &R(-q^{\frac{1}{2}},1)-R(q^{\frac{1}{2}},-1)=\sum_{n=0}^{\infty}\frac{q^{\frac{1}{4}n^2}}{(-q^{\frac{1}{2}};-q^{\frac{1}{2}})_{n}}-\sum_{n=0}^{\infty}\frac{(-1)^{n}q^{\frac{1}{4}n^2}}{(-q^{\frac{1}{2}};-q^{\frac{1}{2}})_{n}}\notag\\
        &=2q^{1/4}\sum_{n=0}^{\infty}\frac{q^{n^2+n}}{(-q^{\frac{1}{2}};-q^{\frac{1}{2}})_{2n+1}}=2q^{1/4}\frac{(q^{1/2},q^{7/2},q^{4};q^{4})_{\infty}}{(q;q)_{\infty}}.
    \end{align}
This proves \eqref{R1+1} and \eqref{R1-1}.

We have
    \begin{align}
        &R(-q^{\frac{1}{2}},q^{-\frac{1}{2}})=\sum_{i,j\ge0}\frac{(-1)^{i}q^{\frac{1}{4}i^2+\frac{1}{2}ij+\frac{1}{4}j^2+\frac{1}{2}i-\frac{1}{2}j}}{(q;q)_i(q;q)_j}=\sum_{n=0}^{\infty}\frac{q^{\frac{1}{4}n^2-\frac{1}{2}n}}{(q;q)_n}\sum_{i=0}^{n}(-1)^{i}q^{i}{ n \brack i},\label{R1+-2}\\
        &R(q^{\frac{1}{2}},-q^{-\frac{1}{2}})=\sum_{i,j\ge0}\frac{(-1)^{j}q^{\frac{1}{4}i^2+\frac{1}{2}ij+\frac{1}{4}j^2+\frac{1}{2}i-\frac{1}{2}j}}{(q;q)_i(q;q)_j}=\sum_{n=0}^{\infty}\frac{(-1)^{n}q^{\frac{1}{4}n^2-\frac{1}{2}n}}{(q;q)_n}\sum_{i=0}^{n}(-1)^{i}q^{i}{   n \brack i}\label{R1-+2}.
    \end{align}
Using \eqref{finite-3} and \eqref{euler-2} we obtain \eqref{R1-3} and \eqref{R1-2} as follows:
    \begin{align*}
        &R(-q^{\frac{1}{2}},q^{-\frac{1}{2}})+R(q^{\frac{1}{2}},-q^{-\frac{1}{2}})=2\sum_{n=0}^{\infty}\frac{q^{n^2-n}}{(q^2;q^2)_{n}}=4(-q^2;q^2)_\infty, \nonumber \\
        &R(-q^{\frac{1}{2}},q^{-\frac{1}{2}})-R(q^{\frac{1}{2}},-q^{-\frac{1}{2}})=2q^{-1/4}\sum_{n=0}^{\infty}\frac{q^{n^2}}{(q^2;q^2)_{n}}=2q^{-1/4}(-q;q^2)_\infty. \qedhere
    \end{align*}
\end{proof}

\begin{proof}[Proof of Theorem \ref{thm12-triple}]
By \eqref{F-00} we have
    \begin{align}
        &\sum_{i,j\ge0}\frac{q^{i^2+2ij+j^2}}{(q;q)_{2i}(q;q)_{2j}}=\frac{1}{4}\big(R(1,1)+R(-1,1)+R(1,-1)+R(-1,-1)\big) \nonumber \\
        &=\frac{1}{4}\Big(\frac{(-q^{1/4},-q^{1/4},q^{1/2};q^{1/2})_\infty}{(q;q)_\infty}+2(-q;q^2)_\infty+\frac{(q^{1/4},q^{1/4},q^{1/2};q^{1/2})_\infty}{(q;q)_\infty}\Big) \nonumber \\
        &=\frac{1}{2}\Big(\frac{(-q,-q,q^{2};q^{2})_\infty}{(q;q)_\infty}+(-q;q^2)_\infty\Big)=\frac{(-q;q^2)_\infty}{2(q;q^{2})_\infty}\Big((-q;q^{2})_\infty+(q;q^{2})_\infty\Big). \label{proof-thm12-triple-last}
    \end{align}
Here for the second equality we used  \eqref{R1++0}, \eqref{R1--0} and \eqref{R1+-0}, and for the third equality we  used \eqref{jac even} with $(a,b)=(1/4,0)$. Now substituting \eqref{ss-1} into \eqref{proof-thm12-triple-last} we obtain \eqref{eq8-1}.

Similarly, applying Lemma \ref{lem-R} and Lemma \ref{lem-mod2} with $(a,b,c)=(1/2,1/2,1/2)$ and $(d,e)\in \{(0,1/2), (1/2,-1/2),(0,0)\}$ and using \eqref{jac even}--\eqref{ss-2} to simplify, we obtain \eqref{eq8-5}--\eqref{eq8-7}.
\end{proof}

\section{Nahm sums for matrix with positive determinant}\label{sec-positive}

\subsection{The cases $({3}/{4},-{1}/{2}, 1)$ and $(1,-1/2,3/4)$}
We find that the Nahm sum $f_{A,B,C,v+L}(q)$ is modular for $A=M(3/4, -1/2, 1)$ and the choices of $(B,C,v,L)$ in Table \ref{tab-6}.

\begin{table}[H]
\caption{Modular quadruples $(A,B,C,v+L)$ with $A=M(3/4,-1/2,1)$} \label{tab-6}
\centering
\renewcommand{\arraystretch}{1.5}
\begin{tabular}{|c|c|c|c|}
  \hline
   $B$ & $C$ & $v$ & $L$ \\ \hline
   $(1/4,-1/2)^\mathrm{T}$ & $1/28$  & $(0,0)^\mathrm{T}$, $(1,0)^\mathrm{T}$ & \multirow{3}{*}{$\mathbb{Z}(2,0)^\mathrm{T}+\mathbb{Z}(0,1)^\mathrm{T}$} \\ \cline{1-3}
   $(0,0)^\mathrm{T}$  & $-3/56$ & $(0,0)^\mathrm{T}$, $(1,0)^\mathrm{T}$ & \\ \cline{1-3}
   $(1/2,0)^\mathrm{T}$ & $ 1/56$ &  $(0,0)^\mathrm{T}$, $(1,0)^\mathrm{T}$ & \\
  \hline
   $(1/4,-1/2)^\mathrm{T}$ & $1/28$ & $(0,0)^\mathrm{T}$, $(0,1)^\mathrm{T}$ & \multirow{3}{*}{$\mathbb{Z}(2,0)^\mathrm{T}+\mathbb{Z}(0,2)^\mathrm{T}$} \\ \cline{1-3}
    $(0,0)^\mathrm{T}$ &  $-3/56$ & $(1,0)^\mathrm{T}$, $(0,1)^\mathrm{T}$  &  \\
    \cline{1-3}
    $(1/2,0)^\mathrm{T}$ & $1/56$ & $(1,0)^\mathrm{T}$, $(1,1)^\mathrm{T}$ &  \\
    \hline
\end{tabular}
\end{table}
The modularity is justified by the theorems below.
\begin{theorem}
    We have
    \begin{align}
        \sum_{i,j\ge0}\frac{q^{\frac{3}{2}i^2-ij+\frac{1}{2}j^2+\frac{1}{2}i-\frac{1}{2}j}}{(q;q)_{2i}(q;q)_{j}}&=\frac{2}{(q;q^2)_{\infty}^2(q^{2},q^{4},q^{10},q^{12};q^{14})_{\infty}},\label{ep5-1}\\
        \sum_{i,j\ge0}\frac{q^{\frac{3}{2}i^2-ij+\frac{1}{2}j^2+\frac{3}{2}i-\frac{1}{2}j}}{(q;q)_{2i+1}(q;q)_{j}}&=\frac{2}{(q;q^2)_{\infty}^2(q^{2},q^{6},q^{8},q^{12};q^{14})_{\infty}},\label{ep5-2}\\
        \sum_{i,j\ge0}\frac{q^{\frac{3}{2}i^2-ij+\frac{1}{2}j^2+\frac{5}{2}i-\frac{1}{2}j}}{(q;q)_{2i+1}(q;q)_{j}}&=\frac{2}{(q;q^2)_{\infty}^2(q^{4},q^{6},q^{8},q^{10};q^{14})_{\infty}},\label{ep5-3}\\
        \sum_{i,j\ge0}\frac{q^{3i^2-2ij+j^2}}{(q^2;q^2)_{2i}(q^2;q^2)_{j}}
        &=\frac{J_2J_{14}^2J_{28}J_{6,28}J_{10,28}}{J_1J_{3,28}J_{7,28}J_{8,28}J_{11,28}J_{12,28}}, \label{ep5-4}\\
        \sum_{i,j\ge0}\frac{q^{3i^2-2ij+j^2+2i}}{(q^2;q^2)_{2i}(q^2;q^2)_{j}}
        &=\frac{J_2J_{14}^2J_{28}J_{2,28}J_{10,28}}{J_1J_{4,28}J_{5,28}J_{7,28}J_{8,28}J_{9,28}}, \label{ep5-5}\\
        \sum_{i,j\ge0}\frac{q^{3i^2-2ij+j^2+4i-2j}}{(q^2;q^2)_{2i+1}(q^2;q^2)_{j}}  &=\frac{J_2J_{14}^2J_{28}J_{2,28}J_{6,28}}{q J_1J_{1,28}J_{4,28}J_{7,28}J_{12,28}J_{13,28}}. \label{ep5-6}
    \end{align}
\end{theorem}
The identities \eqref{ep5-1}--\eqref{ep5-6} were first proved by Wang \cite[Eqs.\ (3.133), (3.139), (3.148)]{Wang-rank2}, and they are utilized to prove the modularity of the full Nahm sums $f_{A,B,C}(q)$ with $(A,B,C)$ in Table \ref{tab-6} (see \cite[Theorem 3.9]{Wang-rank2}).

\begin{theorem}\label{thm34121-sub}
    We have
    \begin{align}
        \sum_{i,j\ge0}\frac{q^{\frac{3}{2}i^2-2ij+2j^2+\frac{1}{2}i-j}}{(q;q)_{2i}(q;q)_{2j}}&=\frac{1}{(q;q^2)_\infty^2(q^{2},q^{4},q^{10},q^{12};q^{14})_\infty},\label{eq5-1}\\
        \sum_{i,j\ge0}\frac{q^{\frac{3}{2}i^2-2ij+2j^2-\frac{1}{2}i+j}}{(q;q)_{2i}(q;q)_{2j+1}}&=\frac{1}{(q;q^2)_\infty^2(q^{2},q^{4},q^{10},q^{12};q^{14})_\infty},\label{eq5-2}\\
        \sum_{i,j\ge0}\frac{q^{\frac{3}{2}i^2-2ij+2j^2+\frac{3}{2}i-j}}{(q;q)_{2i+1}(q;q)_{2j}}&=\frac{1}{(q;q^2)_\infty^2(q^{2},q^{6},q^{8},q^{12};q^{14})_\infty},\label{eq5-3}\\
        \sum_{i,j\ge0}\frac{q^{\frac{3}{2}i^2-2ij+2j^2+\frac{1}{2}i+j}}{(q;q)_{2i+1}(q;q)_{2j+1}}&=\frac{1}{(q;q^2)_\infty^2(q^{2},q^{6},q^{8},q^{12};q^{14})_\infty},\label{eq5-5}\\
        \sum_{i,j\ge0}\frac{q^{\frac{3}{2}i^2-2ij+2j^2+\frac{5}{2}i-j}}{(q;q)_{2i+1}(q;q)_{2j}}&=\frac{1}{(q;q^2)_\infty^2(q^{4},q^{6},q^{8},q^{10};q^{14})_\infty},\label{eq5-4}\\
        \sum_{i,j\ge0}\frac{q^{\frac{3}{2}i^2-2ij+2j^2+\frac{3}{2}i+j}}{(q;q)_{2i+1}(q;q)_{2j+1}}&=\frac{1}{(q;q^2)_\infty^2(q^{4},q^{6},q^{8},q^{10};q^{14})_\infty}.\label{eq5-6}
    \end{align}
\end{theorem}
\begin{proof}
By Lemma \ref{lem-mod2} we have
\begin{align*}
&\sum_{i,j\ge0}\frac{q^{\frac{3}{2}i^2-2ij+2j^2+\frac{1}{2}i-j}}{(q;q)_{2i}(q;q)_{2j}}=\frac{1}{2}\Big(\sum_{i,j\geq 0} \frac{q^{\frac{3}{2}i^2-ij+\frac{1}{2}j^2+\frac{1}{2}i-\frac{1}{2}j}}{(q;q)_{2i}(q;q)_j} + \sum_{i,j\geq 0} \frac{(-1)^jq^{\frac{3}{2}i^2-ij+\frac{1}{2}j^2+\frac{1}{2}i-\frac{1}{2}j}}{(q;q)_{2i}(q;q)_j}\Big), \\
&\sum_{i,j\ge0}\frac{q^{\frac{3}{2}i^2-2ij+2j^2-\frac{1}{2}i+j}}{(q;q)_{2i}(q;q)_{2j+1}}=\frac{1}{2}\Big(\sum_{i,j\geq 0} \frac{q^{\frac{3}{2}i^2-ij+\frac{1}{2}j^2+\frac{1}{2}i-\frac{1}{2}j}}{(q;q)_{2i}(q;q)_j} - \sum_{i,j\geq 0} \frac{(-1)^jq^{\frac{3}{2}i^2-ij+\frac{1}{2}j^2+\frac{1}{2}i-\frac{1}{2}j}}{(q;q)_{2i}(q;q)_j}\Big).
\end{align*}
Summing over $j$ first using \eqref{euler-2}, we see that the second sums in the right side of both identities vanish since
\begin{align}\label{q-product-zero}
(q^{-i};q)_\infty=0, \quad \forall i\geq 0.
\end{align}
Hence using \eqref{ep5-1} we obtain \eqref{eq5-1} and \eqref{eq5-2}.

In the same way we obtain \eqref{eq5-3}--\eqref{eq5-6} from \eqref{ep5-2} and \eqref{ep5-3}.
\end{proof}

\subsection{The cases $(1, -1, 2)$ and $(2,-1,1)$}
We find that the Nahm sum $f_{A,B,C,v+L}(q)$ is modular for $A=M(1,-1,2)$ and the choices of $(B,C,v,L)$ in Table \ref{tab-7}.

\begin{table}[H]
\caption{Modular quadruples $(A,B,C,v+L)$ with $A=M(1,-1,2)$} \label{tab-7}
\centering
\renewcommand{\arraystretch}{1.5}
\begin{tabular}{|c|c|c|c|}
  \hline
   $B$ & $C$ & $v$ & $L$ \\ \hline
   $(-1/2,1)^\mathrm{T}$ & $1/6$  & $(0,0)^\mathrm{T}$, $(1,0)^\mathrm{T}$ & \multirow{2}{*}{$\mathbb{Z}(2,0)^\mathrm{T}+\mathbb{Z}(0,1)^\mathrm{T}$} \\ \cline{1-3}
   $(-3/2,2)^\mathrm{T}$  & $25/24$ & $(0,0)^\mathrm{T}$, $(1,0)^\mathrm{T}$ &  \\
  \hline
\end{tabular}
\end{table}
The modularity of the full Nahm sum $f_{A,B,C}(q)$ with $(A,B,C)$ given in Table \ref{tab-7} were conjectured by Zagier \cite{Zagier} and confirmed in the work of Calinescu–Milas–Penn \cite{CMP} (see also \cite[Theorem 3.3]{Wang-rank2}).

\begin{theorem}
    We have
    \begin{align}
        \sum_{i,j\ge0}\frac{q^{2i^2-2ij+j^2-i+j}}{(q;q)_{2i}(q;q)_{j}}&=\frac{1}{(q;q^2)_{\infty}^2 (q^{2};q^{4})_{\infty}},\label{ep6-1}\\
        \sum_{i,j\ge0}\frac{q^{2i^2-2ij+j^2+i}}{(q;q)_{2i+1}(q;q)_{j}}&=\frac{1}{(q;q^2)_{\infty}^2 (q^{2};q^{4})_{\infty}}, \label{ep6-3} \\
       \sum_{i,j\geq 0} \frac{q^{2i^2-2ij+j^2-3i+2j}}{(q;q)_{2i}(q;q)_j} &=  q^{-1}\frac{(q^{2};q^{4})_{\infty}}{(q;q^2)_{\infty}^3}, \label{add-ep6-4} \\
        \sum_{i,j\ge0}\frac{q^{2i^2-2ij+j^2-i+j}}{(q;q)_{2i+1}(q;q)_{j}}&=\frac{(q^{2};q^{4})_{\infty}}{(q;q^2)_{\infty}^3 }.\label{ep6-2}
    \end{align}
\end{theorem}
\begin{proof}
As in \cite[Eq.\ (3.23)]{Wang-rank2}, we define
    \begin{align}
        &F(u,v)=F(u,v;q):=\sum_{i,j\ge0}\frac{q^{\frac{1}{2}i^2-ij+j^2}u^i v^j}{(q;q)_{i}(q;q)_{j}}=\sum_{j=0}^{\infty}\frac{q^{j^2}v^j}{(q;q)_{j}}\sum_{i=0}^{\infty}\frac{q^{\frac{1}{2}i^2-\frac{1}{2}i}(uq^{-j+\frac{1}{2}})^i}{(q;q)_{i}}\notag\\
        &=\sum_{j=0}^{\infty}\frac{q^{j^2}v^j(-uq^{-j+\frac{1}{2}};q)_{\infty}}{(q;q)_{j}}.\quad \text{(by \eqref{euler-2})}\label{f-3.6}
    \end{align}
It is known that (see e.g.\ \cite[Eqs.\ (3.18), (3.20)]{Wang-rank2})
\begin{align}
        &F(q^{-\frac{1}{2}},q)=\frac{2}{(q;q^2)_\infty^2 (q^2;q^4)_\infty}, \label{add-3.5-1} \\
        &F(q^{-\frac{3}{2}},q^2)=\frac{2q^{-1}(q^{2};q^{4})_{\infty}}{(q;q^2)_{\infty}^3}. \label{add-3.5-2}
    \end{align}
From \eqref{f-3.6} and \eqref{q-product-zero} we deduce that
    \begin{align}
        F(-q^{-\frac{1}{2}},q)=0,\quad\quad
        F(-q^{-\frac{3}{2}},q^2)=0. \label{add-3.5-3}
    \end{align}
    Using \eqref{F-add} and \eqref{F-subtract}, we have
    \begin{align}
        \sum_{i,j\ge0}\frac{q^{2i^2-2ij+j^2-i+j}}{(q;q)_{2i}(q;q)_{j}}&=\frac{1}{2}\left(F(q^{-\frac{1}{2}},q)+F(-q^{-\frac{1}{2}},q)\right), \label{add-3.5-proof-1} \\
        \sum_{i,j\ge0}\frac{q^{2i^2-2ij+j^2+i}}{(q;q)_{2i+1}(q;q)_{j}}    &=\frac{1}{2}\left(F(q^{-\frac{1}{2}},q)-F(-q^{-\frac{1}{2}},q)\right), \label{add-3.5-proof-3} \\
         \sum_{i,j\geq 0} \frac{q^{2i^2-2ij+j^2-3i+2j}}{(q;q)_{2i}(q;q)_j} &=  \frac{1}{2}\left(F(q^{-\frac{3}{2}},q^2)+F(-q^{-\frac{3}{2}},q^2)\right), \label{add-3.5-proof-2‘} \\
        \sum_{i,j\ge0}\frac{q^{2i^2-2ij+j^2-i+j}}{(q;q)_{2i+1}(q;q)_{j}}
        &=\frac{q}{2}\left(F(q^{-\frac{3}{2}},q^2)-F(-q^{-\frac{3}{2}},q^2)\right). \label{add-3.5-proof-2}
    \end{align}
Substituting \eqref{add-3.5-1}--\eqref{add-3.5-3} into \eqref{add-3.5-proof-1}--\eqref{add-3.5-proof-2}, we obtain \eqref{ep6-1}--\eqref{ep6-2}.
\end{proof}

\subsection{The case $(1, -1/2, 1)$}
In this subsection we let $A=M(1, -1/2, 1)$. We find that the Nahm sum $f_{A,B,C,v+L}(q)$ is modular for the choices of $(B,C,v,L)$ in Table \ref{tab-8} where $v_1,v_2$ can take arbitrary integer values.

\begin{table}[htbp]
\caption{Modular quadruples $(A,B,C,v+L)$ with $A=M(1,-1/2,1)$} \label{tab-8}
\centering
\renewcommand{\arraystretch}{1.5}
\begin{tabular}{|c|c|c|c|}
  \hline
   $B$ & $C$ & $v$ & $L$ \\ \hline
   $(0,-1/2)^\mathrm{T}$ & $1/20$  & $(0,0)^\mathrm{T}$ & \multirow{3}{*}{$\mathbb{Z}(2,0)^\mathrm{T}+\mathbb{Z}(0,1)^\mathrm{T}$} \\ \cline{1-3}
   $(-1/2,0)^\mathrm{T}$  & $1/20$ & $(0,0)^\mathrm{T}$, $(1,0)^\mathrm{T}$ &  \\ \cline{1-3}
   $(0,0)^\mathrm{T}$ & $-1/20$ & $(0,0)^\mathrm{T}$, $(1,0)^\mathrm{T}$ & \\ \hline
   $(0,-1/2)^\mathrm{T}$ & $1/20$  & $(v_1,v_2)^\mathrm{T}$  & \multirow{3}{*}{$\mathbb{Z}(2,0)^\mathrm{T}+\mathbb{Z}(0,2)^\mathrm{T}$} \\ \cline{1-3}
    $(-1/2,0)^\mathrm{T}$  & $1/20$ & $(v_1,v_2)^\mathrm{T}$ &  \\ \cline{1-3}
   $(0,0)^\mathrm{T}$ & $-1/20$ &  $(v_1,v_2)^\mathrm{T}$   & \\
  \hline
\end{tabular}
\end{table}
Similar to Section \ref{subsec-111}, if we interchange the coordinates in $B,v$ in the first three rows, we obtain modualr quadruples for the lattice $L=\mathbb{Z}(1,0)^\mathrm{T}+\mathbb{Z}(0,2)^\mathrm{T}$. For instance, from the first row of Table \ref{tab-8} we get the following modular quadruple for free:
\begin{align}
    B=(-1/2,0)^\mathrm{T}, \quad C=1/20, \quad v=(0,0)^\mathrm{T}, \quad L=\mathbb{Z}(1,0)^\mathrm{T}+\mathbb{Z}(0,2)^\mathrm{T}.
\end{align}

\begin{theorem}
    We have
    \begin{align}
        \sum_{i,j\geq 0} \frac{q^{2i^2-ij+\frac{1}{2}j^2-\frac{1}{2}j}}{(q;q)_{2i}(q;q)_{j}}&=\frac{2}{(q;q^2)_{\infty}^2(q^{2},q^{8};q^{10})_{\infty}},\label{ep7-1}\\
        \sum_{i,j\geq 0} \frac{q^{2i^2-ij+\frac{1}{2}j^2+i-\frac{1}{2}j}}{(q;q)_{2i+1}(q;q)_{j}}&=\frac{2}{(q;q^2)_{\infty}^2(q^{2},q^{8};q^{10})_{\infty}},\label{ep7-2}\\
        \sum_{i,j\geq 0} \frac{q^{2i^2-ij+\frac{1}{2}j^2+2i-\frac{1}{2}j}}{(q;q)_{2i+1}(q;q)_{j}}&=\frac{2}{(q;q^2)_{\infty}^2(q^{4},q^{6};q^{10})_{\infty}},\label{ep7-3}\\
        \sum_{i,j\geq 0} \frac{q^{4i^2-2ij+j^2-2i}}{(q^2;q^2)_{2i}(q^2;q^2)_{j}}&=\frac{J_{2}^2J_{10}J_{3,10}J_{2,20}J_{6,20}}{J_{1}^2J_{20}^3J_{8,20}},\label{ep7-4}\\
        \sum_{i,j\geq 0} \frac{q^{4i^2-2ij+j^2}}{(q^2;q^2)_{2i}(q^2;q^2)_{j}}&=\frac{J_{2}^2J_{10}J_{1,10}J_{2,20}J_{6,20}}{J_{1}^2J_{20}^3J_{4,20}}.\label{ep7-5}
    \end{align}
\end{theorem}
\begin{proof}
Summing over $j$ first using \eqref{euler-2} and then using \eqref{trans-g}, we have
    \begin{align*}
        &\sum_{i,j\ge0}\frac{q^{2i^2-ij+\frac{1}{2}j^2+ti-\frac{1}{2}j}}{(q;q)_{2i+s}(q;q)_{j}}=\sum_{i=0}^{\infty}\frac{q^{2i^2+ti}}{(q;q)_{2i+s}}(-q^{-i},q)_\infty =2(-q,q)_\infty \sum_{i=0}^{\infty}\frac{q^{\frac{3}{2}i^2+(t-\frac{1}{2})i}(-q,q)_i}{(q;q)_{2i+s}}.
    \end{align*}
For $(s,t)=(0,0)$, $(1,1)$ and $(1,2)$,  using \eqref{s-46}, \eqref{s-62} and \eqref{s-63} we obtain \eqref{ep7-1}--\eqref{ep7-3}, respectively.

Similarly, we have
    \begin{align*}
        &\sum_{i,j\ge0}\frac{q^{4i^2-2ij+j^2+ti}}{(q^2;q^2)_{2i}(q^2;q^2)_{j}}
        =\sum_{i=0}^{\infty}\frac{q^{4i^2+ti}}{(q^2;q^2)_{2i}}(-q^{-2i+1};q^2)_\infty
        =(-q;q^2)_\infty\sum_{i=0}^{\infty}\frac{q^{3i^2+ti}(-q;q^2)_i}{(q^2;q^2)_{2i}}.
    \end{align*}
For $t=-2,0$, using \eqref{s-95} and \eqref{s-100c} we obtain \eqref{ep7-4} and \eqref{ep7-5}, respectively.
\end{proof}

\begin{theorem}
    We have
    \begin{align}
        \sum_{i,j\ge0}\frac{q^{2i^2-2ij+2j^2-j}}{(q;q)_{2i}(q;q)_{2j}}&=\frac{1}{(q;q^2)_\infty^2(q^{2},q^{8};q^{10})_\infty},\label{eq6-1}\\
        \sum_{i,j\ge0}\frac{q^{2i^2-2ij+2j^2-i+j}}{(q;q)_{2i}(q;q)_{2j+1}}&=\frac{1}{(q;q^2)_\infty^2(q^{2},q^{8};q^{10})_\infty},\label{eq6-2}\\
        \sum_{i,j\ge0}\frac{q^{2i^2-2ij+2j^2+i}}{(q;q)_{2i+1}(q;q)_{2j+1}}&=\frac{1}{(q;q^2)_\infty^2(q^{2},q^{8};q^{10})_\infty},\label{eq6-4}\\
        \sum_{i,j\ge0}\frac{q^{2i^2-2ij+2j^2+2i-j}}{(q;q)_{2i+1}(q;q)_{2j}}&=\frac{1}{(q;q^2)_\infty^2(q^{4},q^{6};q^{10})_\infty},\label{eq6-3}\\
        \sum_{i,j\ge0}\frac{q^{2i^2-2ij+2j^2+i+j}}{(q;q)_{2i+1}(q;q)_{2j+1}}&=\frac{1}{(q;q^2)_\infty^2(q^{4},q^{6};q^{10})_\infty}.\label{eq6-5}
    \end{align}
\end{theorem}
\begin{proof}
This can be proved using \eqref{ep7-1}--\eqref{ep7-3} in the same way as the proof of Theorem \ref{thm34121-sub}. We omit the details.
\end{proof}

The full Nahm sums $f_{A,B_i,C_i}(q)$ with
\begin{align}
    A=M(1,-1/2,1), ~~ B_1=(0,0)^\mathrm{T}, ~~B_2=(-1/2,0)^\mathrm{T}, \nonumber \\
    ~~ B_3=(0,-1/2)^\mathrm{T}, ~~ C_1=-1/20, ~~C_2=C_3=1/20
\end{align}
were proved to be modular by Vlasenko and Zwegers  \cite[Theorem 3.2]{VZ}. Let $L=\mathbb{Z}(2,0)^\mathrm{T}+\mathbb{Z}(0,2)^\mathrm{T}$. From \eqref{eq6-1}--\eqref{eq6-4} we know that $f_{A,B_3,C_3,v+L}(q)$ is modular for $v\in \{(0,0)^\mathrm{T},(0,1)^\mathrm{T},(1,1)^\mathrm{T}\}$, and therefore
\begin{align}\label{full-part}
    f_{A,B_3,C_3,(0,1)^\mathrm{T}+L}(q)&=f_{A,B_3,C_3}(q)-f_{A,B_3,C_3,(0,0)^\mathrm{T}+L}(q)-f_{A,B_3,C_3,(1,0)^\mathrm{T}+L}(q)  \nonumber \\
    &\quad -f_{A,B_3,C_3,(1,1)^\mathrm{T}}(q)
\end{align}
is also modular.

Similarly,  \eqref{eq6-3} and \eqref{eq6-5} implies that $f_{A,B_1,C,v+L}(q)$ is modular for $v$ being $(1,0)^\mathrm{T}$ or $(1,1)^\mathrm{T}$. Interchanging $i$ with $j$ in \eqref{eq6-3} we know that it is also modular for $v=(0,1)^\mathrm{T}$. Hence for the reason similar to \eqref{full-part}, we know that it is modular for $v=(0,0)^\mathrm{T}$.

\subsection{The cases $(3/2,1,2)$ and $(2,1,3/2)$}
We find that the Nahm sum $f_{A,B,C,v+L}(q)$ is modular for $A=M(3/2,1,2)$ and the choices of $(B,C,v,L)$ in Table \ref{tab-10}.

\begin{table}[H]
\caption{Modular quadruples $(A,B,C,v+L)$ with $A=M(3/2,1,2)$} \label{tab-10}
\centering
\renewcommand{\arraystretch}{1.5}
\begin{tabular}{|c|c|c|c|}
  \hline
   $B$ & $C$ & $v$ & $L$ \\ \hline
   $(-1/2,0)^\mathrm{T}$ & $1/168$  & $(0,0)^\mathrm{T}$, $(1,0)^\mathrm{T}$ & \multirow{3}{*}{$\mathbb{Z}(2,0)^\mathrm{T}+\mathbb{Z}(0,1)^\mathrm{T}$} \\ \cline{1-3}
   $(0,0)^\mathrm{T}$  & $-5/168$ & $(0,0)^\mathrm{T}$, $(1,0)^\mathrm{T}$ &  \\ \cline{1-3}
   $(1/2,1)^\mathrm{T}$ & $25/168$ & $(0,0)^\mathrm{T}$, $(1,0)^\mathrm{T}$ & \\ \hline
\end{tabular}
\end{table}
The full Nahm sums $f_{A,B,C}(q)$ with $(A,B,C)$ given in Table \ref{tab-10} are on Zagier's list \cite[Table 2]{Zagier} and their modularity were confirmed by Wang \cite[Theorem 3.8]{Wang-rank2}.

\begin{theorem}
    We have
    \begin{align}
        \sum_{i,j\geq 0} \frac{q^{3i^2+2ij+j^2-i}}{(q;q)_{2i}(q;q)_{j}}&=\frac{(q^{3},q^{11},q^{14};q^{14})_{\infty}(q^{8},q^{20};q^{28})_{\infty}}{(q;q)_{\infty}},\label{ep10-1}\\
        \sum_{i,j\geq 0} \frac{q^{3i^2+2ij+j^2+2i+j}}{(q;q)_{2i+1}(q;q)_{j}}&=\frac{(q^{4},q^{10},q^{14};q^{14})_{\infty}(q^{6},q^{22};q^{28})_{\infty}}{(q;q)_{\infty}},\label{ep10-4}\\
        \sum_{i,j\geq 0} \frac{q^{3i^2+2ij+j^2}}{(q;q)_{2i}(q;q)_{j}}&=\frac{(q^{2},q^{12},q^{14};q^{14})_{\infty}(q^{10},q^{18};q^{28})_{\infty}}{(q;q)_{\infty}},\label{ep10-2}\\
        \sum_{i,j\geq 0} \frac{q^{3i^2+2ij+j^2+3i+j}}{(q;q)_{2i+1}(q;q)_{j}}&=\frac{(q^{5},q^{9},q^{14};q^{14})_{\infty}(q^{4},q^{24};q^{28})_{\infty}}{(q;q)_{\infty}},\label{ep10-5}\\
        \sum_{i,j\geq 0} \frac{q^{3i^2+2ij+j^2+i+j}}{(q;q)_{2i}(q;q)_{j}}&=\frac{(q,q^{13},q^{14};q^{14})_{\infty}(q^{12},q^{16};q^{28})_{\infty}}{(q;q)_{\infty}},\label{ep10-3}\\
        \sum_{i,j\geq 0} \frac{q^{3i^2+2ij+j^2+4i+2j}}{(q;q)_{2i+1}(q;q)_{j}}&=\frac{(q^{6},q^{8},q^{14};q^{14})_{\infty}(q^{2},q^{26};q^{28})_{\infty}}{(q;q)_{\infty}}.\label{ep10-6}
    \end{align}
\end{theorem}
These identities have essentially been proved in Wang's work \cite[Theorem 3.8]{Wang-rank2}, although they were not explicitly stated.
\begin{proof}
Let
    \begin{align}
        F(u,v)=F(u,v;q):=\sum_{i,j\ge0}\frac{q^{3i^2+4ij+4j^2}u^iv^j}{(q^4;q^4)_{i}(q^4;q^4)_{j}}.
    \end{align}
From \cite[Proof of Theorem 3.8]{Wang-rank2} we have
    \begin{align}
        &F(q^{-2},1)=\sum_{i,j\ge0}\frac{q^{(i+2j)^{2}+4j}}{(q^4;q^4)_{i}(q^8;q^8)_{j}}=R_0(q)+R_1(q), \label{R0R1-id}\\
        &F(1,1)=\sum_{i,j\ge0}\frac{q^{(i+2j)^{2}+2i}}{(q^4;q^4)_{i}(q^8;q^8)_{j}}=S_0(q)+S_1(q), \label{S0S1-id}\\
        &F(q^2,q^4)=\sum_{i,j\ge0}\frac{q^{(i+2j)^{2}+2i}}{(q^4;q^4)_{i}(q^8;q^8)_{j}}=T_0(q)+T_1(q). \label{T0T1-id}
    \end{align}
Here $R_k(q), S_k(q)$ and $T_k(q)$ ($k=0,1$) correspond to the sums in the middle of each identity with restriction $i\equiv k$ (mod 2). Wang \cite[Eqs.\ (3.84), (3.85), (3.90), (3.91), (3.96), (3.97)]{Wang-rank2} proved that
\begin{align}
&R_0(q^{\frac{1}{4}})=\frac{J_{14}J_{3,28}J_{8,28}J_{11,28}}{J_1J_{28}^3}, \quad R_1(q^{\frac{1}{4}})=q^{\frac{1}{4}}\frac{J_{14}J_{4,28}J_{6,28}J_{10,28}}{J_1J_{28}^3},  \label{exam8-1-R0-result} \\
&S_0(q^{\frac{1}{4}})=\frac{J_{14}J_{2,28}J_{10,28}J_{12,28}}{J_1J_{28}^3}, \quad S_1(q^{\frac{1}{4}})=q^{\frac{3}{4}}\frac{J_{14}J_{4,28}J_{5,28}J_{9,28}}{J_1J_{28}^3}, \label{exam8-2-S0-result} \\
&T_0(q^{\frac{1}{4}})=\frac{J_{14}J_{1,28}J_{12,28}J_{13,28}}{J_1J_{28}^3}, \quad T_1(q^{\frac{1}{4}})=q^{\frac{5}{4}}\frac{J_{14}J_{2,28}J_{6,28}J_{8,28}}{J_1J_{28}^3}. \label{exam8-3-T0-result}
\end{align}
Note that both the parity of $3i^2+4ij+4j^2-2i$ and $(i+2j)^2+4j$ are the same with $i$. Hence considering 2-dissections on both sides of  \eqref{R0R1-id}, we deduce that
\begin{align*}
\sum_{\begin{smallmatrix} i,j\ge0 \\ i\equiv k \!\!\! \pmod{2} \end{smallmatrix}}\frac{q^{3i^2+4ij+4j^2-2i}}{(q^4;q^4)_{i}(q^4;q^4)_{j}} =\sum_{\begin{smallmatrix} i,j\geq 0 \\ i\equiv k \!\!\! \pmod{2} \end{smallmatrix}} \frac{q^{(i+2j)^2+4j}}{(q^4;q^4)_{i}(q^8;q^8)_{j}}=R_k(q), \quad k=0,1.
\end{align*}
Now using \eqref{exam8-1-R0-result} we obtain \eqref{ep10-1} and \eqref{ep10-4}.

In the same way, using \eqref{S0S1-id} and \eqref{exam8-2-S0-result} we obtain \eqref{ep10-2} and \eqref{ep10-5}, using \eqref{T0T1-id} and \eqref{exam8-3-T0-result} we obtain \eqref{ep10-3} and \eqref{ep10-6}.
\end{proof}

\subsection{The cases $(2,1,1)$ and $(1,1,2)$}
We find that the Nahm sum $f_{A,B,C,v+L}(q)$ is modular for $A=M(2,1,1)$ and the choices of $(B,C,v,L)$ in Table \ref{tab-11}.
\begin{table}[H]
\caption{Modular quadruples $(A,B,C,v+L)$ with $A=M(2,1,1)$} \label{tab-11}
\centering
\renewcommand{\arraystretch}{1.5}
\begin{tabular}{|c|c|c|c|}
  \hline
   $B$ & $C$ & $v$ & $L$ \\ \hline
   $(-1,1/2)^\mathrm{T}$ & $1/8$  & $(0,0)^\mathrm{T}$, $(1,0)^\mathrm{T}$ & \multirow{2}{*}{$\mathbb{Z}(2,0)^\mathrm{T}+\mathbb{Z}(0,1)^\mathrm{T}$} \\ \cline{1-3}
   $(0,1/2)^\mathrm{T}$  & $0$ & $(0,0)^\mathrm{T}$, $(1,0)^\mathrm{T}$ &  \\ \hline
    $(0,0)^\mathrm{T}$ & $-1/32$  & $(0,0)^\mathrm{T}$, $(0,1)^\mathrm{T}$ & \multirow{3}{*}{$\mathbb{Z}(1,0)^\mathrm{T}+\mathbb{Z}(0,2)^\mathrm{T}$} \\ \cline{1-3}
   $(0,1/2)^\mathrm{T}$  & $0$ & $(0,0)^\mathrm{T}$, $(0,1)^\mathrm{T}$ &  \\ \cline{1-3}
   $(1,1)^\mathrm{T}$ & $7/32$ & $(0,0)^\mathrm{T}$, $(0,1)^\mathrm{T}$ & \\ \hline
\end{tabular}
\end{table}

The full Nahm sums $f_{A,B,C}(q)$ with $(A,B,C)$ given in Table \ref{tab-11} are on Zagier's list \cite[Table 1]{Zagier}, and their modularity are known (see \cite[Theorem 4.2]{VZ} and \cite[Theorem 3.2]{Wang-rank2}).
\begin{theorem}
    We have
    \begin{align}
        \sum_{i,j\geq 0}\frac{q^{4i^2+2ij+\frac{1}{2}j^2-2i+\frac{1}{2}j}}{(q;q)_{2i}(q;q)_{j}}&=\frac{1}{(q;q^2)_\infty(q^2;q^4)_\infty},\label{ep11-1}\\
         \sum_{i,j\geq 0} \frac{q^{4i^2+2ij+\frac{1}{2}j^2+2i+\frac{3}{2}j}}{(q;q)_{2i+1}(q;q)_{j}}&=\frac{1}{(q;q^2)_\infty(q^2;q^4)_\infty},\label{ep11-4}\\
        \sum_{i,j\geq 0} \frac{q^{4i^2+2ij+\frac{1}{2}j^2+\frac{1}{2}j}}{(q;q)_{2i}(q;q)_{j}}&=\frac{(q^{2},q^{14},q^{16};q^{16})_{\infty}(q^{12},q^{20};q^{32})_{\infty}}{(q;q)_{\infty}},\label{ep11-2}\\
        \sum_{i,j\geq 0} \frac{q^{4i^2+2ij+\frac{1}{2}j^2+4i+\frac{3}{2}j}}{(q;q)_{2i+1}(q;q)_{j}}&=\frac{(q^{6},q^{10},q^{16};q^{16})_{\infty}(q^{4},q^{28};q^{32})_{\infty}}{(q;q)_{\infty}}.\label{ep11-5}
    \end{align}
\end{theorem}
\begin{proof}
    For $t\in \{0,1\}$ we define
    \begin{align}
        &F_t(u,v;q):=\sum_{i,j\geq 0} \frac{q^{4i^2+2ij+\frac{1}{2}j^2}u^{i}v^{j}}{(q;q)_{2i+t}(q;q)_{j}}=\sum_{i=0}^{\infty}\frac{q^{4i^2}u^{i}}{(q;q)_{2i+t}}\sum_{j=0}^{\infty}\frac{q^{\frac{1}{2}j^2-\frac{1}{2}j}(vq^{2i+\frac{1}{2}})^j}{(q;q)_{j}}\notag\\
        &=\sum_{i=0}^{\infty}\frac{q^{4i^2}u^i(-vq^{2i+\frac{1}{2}};q)_{\infty}}{(q;q)_{2i+t}}=(-vq^{\frac{1}{2}};q)_\infty \sum_{i=0}^\infty \frac{q^{4i^2}u^i}{(q;q)_{2i+t}(-vq^{\frac{1}{2}};q)_{2i}}. \label{Ft-3.11}
    \end{align}

    From \eqref{Ft-3.11} and \eqref{s-39}--\eqref{S9} we have
    \begin{align*}
        &F_0(q^{-2},q^{\frac{1}{2}})
        =(-q;q)_\infty\sum_{i=0}^{\infty}\frac{q^{4i^2-2i}}{(q^2;q^2)_{2i}}=(-q;q)_\infty(-q^2;q^2)_\infty, \\
        &F_1(q^{2},q^{\frac{3}{2}})=(-q;q)_\infty\sum_{i=0}^{\infty}\frac{q^{4i^2+2i}}{(q^2;q^2)_{2i+1}}=(-q;q)_\infty(-q^2;q^2)_\infty, \\
        &F_0(1,q^{\frac{1}{2}})=(-q;q)_\infty\sum_{i=0}^{\infty}\frac{q^{4i^2}}{(q^2;q^2)_{2i}}=(-q;q)_\infty\frac{(-q^{6},-q^{10},q^{16};q^{16})_{\infty}}{(q^{4};q^{4})_{\infty}}, \\
        &F(q^4,q^{\frac{3}{2}})=(-q;q)_\infty\sum_{i=0}^{\infty}\frac{q^{4i^2+4i}}{(q^2;q^2)_{2i+1}}=(-q;q)_\infty\frac{(-q^{2},-q^{14},q^{16};q^{16})_{\infty}}{(q^{4};q^{4})_{\infty}}.
    \end{align*}
This proves  \eqref{ep11-1}--\eqref{ep11-5}.
\end{proof}
We may also prove some companion identities with non-modular product side. For instance, letting $t=1$ and $(u,v)=(q^{-2},q^{\frac{3}{2}})$ in \eqref{Ft-3.11}  we can prove that
\begin{align}
 \sum_{i,j\geq 0} \frac{q^{4i^2+2ij+\frac{1}{2}j^2-2i+\frac{3}{2}j}}{(q;q)_{2i+1}(q;q)_{j}}&=\frac{1+q^2}{(q;q^2)_\infty(q^2;q^4)_\infty}. \label{ep11-3}
\end{align}

\begin{theorem}
    We have
    \begin{align}
        \sum_{i,j\geq 0} \frac{q^{i^2+2ij+2j^2}}{(q;q)_{i}(q;q)_{2j}}&=\frac{(q^{3},q^{13},q^{16};q^{16})_{\infty}(q^{10},q^{22};q^{32})_{\infty}}{(q;q)_{\infty}}, \label{ep9-1}\\
        \sum_{i,j\geq 0}\frac{q^{i^2+2ij+2j^2+i+2j}}{(q;q)_{i}(q;q)_{2j+1}}&=\frac{(q^{5},q^{11},q^{16};q^{16})_{\infty}(q^{6},q^{26};q^{32})_{\infty}}{(q;q)_{\infty}}, \label{ep9-4}\\
        \sum_{i,j\geq 0}\frac{q^{i^2+2ij+2j^2+j}}{(q;q)_{i}(q;q)_{2j}}&=\frac{(q^{2},q^{14},q^{16};q^{16})_{\infty}(q^{12},q^{20};q^{32})_{\infty}}{(q;q)_{\infty}}, \label{ep9-2}\\
         \sum_{i,j\geq 0}\frac{q^{i^2+2ij+2j^2+i+3j}}{(q;q)_{i}(q;q)_{2j+1}}&=\frac{(q^{6},q^{10},q^{16};q^{16})_{\infty}(q^{4},q^{28};q^{32})_{\infty}}{(q;q)_{\infty}}, \label{ep9-5}\\
        \sum_{i,j\geq 0} \frac{q^{i^2+2ij+2j^2+i+2j}}{(q;q)_{i}(q;q)_{2j}}&= \frac{(q,q^{15},q^{16};q^{16})_{\infty}(q^{14},q^{18};q^{32})_{\infty}}{(q;q)_{\infty}}, \label{ep9-3}\\
        \sum_{i,j\geq 0}\frac{q^{i^2+2ij+2j^2+2i+4j}}{(q;q)_{i}(q;q)_{2j+1}}&=\frac{(q^{7},q^{9},q^{16};q^{16})_{\infty}(q^{2},q^{30};q^{32})_{\infty}}{(q;q)_{\infty}}. \label{ep9-6}
    \end{align}
\end{theorem}
\begin{proof}
As in \cite[Eq.\ (3.9)]{Wang-rank2}, we define
    \begin{align}
        &F(u,v)=F(u,v;q):=\sum_{i,j\geq 0} \frac{q^{i^2+ij+\frac{1}{2}j^2}u^i v^j}{(q;q)_{i}(q;q)_{j}}=\sum_{i=0}^{\infty}\frac{q^{i^2}u^i}{(q;q)_{i}}\sum_{j=0}^{\infty}\frac{q^{\frac{1}{2}j^2-\frac{1}{2}j}(vq^{i+\frac{1}{2}})^j}{(q;q)_{j}}\notag\\
        &=\sum_{i=0}^{\infty}\frac{q^{i^2}u^i(-vq^{i+\frac{1}{2}};q)_{\infty}}{(q;q)_{i}}=(-q^{\frac{1}{2}}v;q)_\infty \sum_{i=0}^\infty \frac{q^{i^2}u^i}{(q,-q^{\frac{1}{2}}v;q)_i}. \quad \text{(by \eqref{euler-2})} \label{f-3.8}
    \end{align}
From \cite[Eqs.\ (3.5), (3.6), (3.8)]{Wang-rank2} we know that
    \begin{align}
        &F(1,1)=\frac{(-q^{1/2};q)_{\infty}(q^{3/2},q^{5/2},q^{4};q^{4})_{\infty}}{(q;q)_{\infty}}, \label{add-3.8-1} \\
        &F(1,q^{\frac{1}{2}})=(-q;q)_{\infty}(-q;q^2)_{\infty}, \label{add-3.8-3} \\
         &F(q,q)=\frac{(-q^{1/2};q)_{\infty}(q^{1/2},q^{7/2},q^{4};q^{4})_{\infty}}{(q;q)_{\infty}}. \label{add-3.8-5}
    \end{align}
Recall the definition in \eqref{f-3.8}.    Replacing $q^{1/2}$ by $-q^{1/2}$ in \eqref{add-3.8-1} and \eqref{add-3.8-5}, we deduce that
    \begin{align}
       & F(1,-1)=\frac{(q^{1/2};q)_{\infty}(-q^{3/2},-q^{5/2},q^{4};q^{4})_{\infty}}{(q;q)_{\infty}},  \label{add-3.8-2} \\
       &F(q,-q)=\frac{(q^{1/2};q)_{\infty}(-q^{1/2},-q^{7/2},q^{4};q^{4})_{\infty}}{(q;q)_{\infty}}. \label{add-3.8-6}
    \end{align}

    Letting $(u,v)=(1,-q^{\frac{1}{2}})$ in \eqref{f-3.8} and using \eqref{cauchy} with $z=q$, we have
    \begin{align}
        F(1,-q^{\frac{1}{2}})=(q;q)_{\infty}\sum_{j=0}^{\infty}\frac{q^{j^2}}{(q;q)_{j}^{2}}=1. \label{add-3.8-4}
    \end{align}

Using \eqref{F-add}, \eqref{add-3.8-1} and \eqref{add-3.8-2} we have
      \begin{align*}
&\sum_{i,j\ge0}\frac{q^{i^2+2ij+2j^2}}{(q;q)_{i}(q;q)_{2j}}=\frac{1}{2}\left(F(1,1)+F(1,-1)\right) \\
&=\frac{1}{2}\frac{(q^3,q^5;q^8)_\infty}{(q;q)_\infty}\Big((-q^{1/2},-q^{7/2},q^4;q^4)_\infty+(q^{1/2},q^{7/2},q^4;q^4)_\infty\Big) \nonumber \\
&=\frac{(q^3,q^5;q^8)_\infty(-q^5,-q^{11},q^{16};q^{16})_\infty}{(q;q)_\infty}. \quad \text{(by \eqref{jac even})}
\end{align*}
This proves \eqref{ep9-1}. Similarly, using \eqref{F-add} and \eqref{F-subtract} we have
 \begin{align*}
         &\sum_{i,j\ge0}\frac{q^{i^2+2ij+2j^2+i+2j}}{(q;q)_{i}(q;q)_{2j+1}}=\frac{1}{2}q^{-1/2}\left(F(1,1)-F(1,-1)\right), \\
         &\sum_{i,j\ge0}\frac{q^{i^2+2ij+2j^2+j}}{(q;q)_{i}(q;q)_{2j}}=\frac{1}{2}\left(F(1,q^{{1}/{2}})+F(1,-q^{{1}/{2}})\right), \\
         &\sum_{i,j\ge0}\frac{q^{i^2+2ij+2j^2+i+3j}}{(q;q)_{i}(q;q)_{2j+1}}=\frac{1}{2}q^{-1}\left(F(1,q^{{1}/{2}})-F(1,-q^{{1}/{2}})\right), \\
         &\sum_{i,j\ge0}\frac{q^{i^2+2ij+2j^2+i+2j}}{(q;q)_{i}(q;q)_{2j}}=\frac{1}{2}\left(F(q,q)+F(q,-q)\right), \\
         &\sum_{i,j\ge0}\frac{q^{i^2+2ij+2j^2+2i+4j}}{(q;q)_{i}(q;q)_{2j+1}}=\frac{1}{2}q^{-3/2}\left(F(q,q)-F(q,-q)\right).
    \end{align*}
Substituting \eqref{add-3.8-1}--\eqref{add-3.8-6} into the above identities and using \eqref{jac even}--\eqref{ss-2} to simplify, we obtain  \eqref{ep9-4}--\eqref{ep9-6}.
\end{proof}

\subsection{The cases $(1/2, -1/2,1)$ and $(1,-1/2,1/2)$}
The full Nahm sums $f_{A,B,C}(q)$ with $(A,B,C)$ in Table \ref{tab-12} are in Zagier's list \cite[Table 2]{Zagier}, and their modularity were confirmed by Wang \cite[Theorem 3.7]{Wang-rank2}. Along the process of proving the modularity, Wang \cite[Eqs.\ (3.52), (3.56), (3.60), (3.63), (3.68), (3.70)]{Wang-rank2} established the following identities:
\begin{align}
        \sum_{i,j\ge0}\frac{q^{i^2-ij+\frac{1}{2}j^2+i-\frac{1}{2}j}}{(q;q)_{2i}(q;q)_{j}}&=2\frac{J_2J_{14}J_{6,14}}{J_1^2J_{3,14}},\label{ep3-1}\\
        \sum_{i,j\ge0}\frac{q^{2i^2-2ij+j^2+4i-2j}}{(q^2;q^2)_{2i+1}(q^2;q^2)_{j}}&=\frac{J_2^3J_{14}J_{5,28}J_{9,28}}{q J_1^2J_4J_{28}J_{8,28}J_{12,28}}, \label{ep3-6} \\
        \sum_{i,j\ge0}\frac{q^{2i^2-2ij+j^2}}{(q^2;q^2)_{2i}(q^2;q^2)_{j}}&=\frac{J_2^3J_{28}J_{3,14}}{J_1^{2} J_4 J_{4,28} J_{12,28}},\label{ep3-4}\\
        \sum_{i,j\ge0}\frac{q^{i^2-ij+\frac{1}{2}j^2+i-\frac{1}{2}j}}{(q;q)_{2i+1}(q;q)_{j}}&=2\frac{J_2J_{14}J_{2,14}}{J_1^2J_{1,14}},\label{ep3-2}\\
          \sum_{i,j\ge0}\frac{q^{2i^2-2ij+j^2+2i}}{(q^2;q^2)_{2i}(q^2;q^2)_{j}}&=\frac{J_2^3J_{14}J_{1,28}J_{13,28}}{J_1^2J_4J_{28}J_{4,28}J_{8,28}}, \label{ep3-5} \\
        \sum_{i,j\ge0}\frac{q^{i^2-ij+\frac{1}{2}j^2+2i-\frac{1}{2}j}}{(q;q)_{2i+1}(q;q)_{j}}&=2\frac{J_2J_{14}J_{4,14}}{J_1^2J_{5,14}}. \label{ep3-3}
    \end{align}
Together with the following identities, we know that
the Nahm sum $f_{A,B,C,v+L}(q)$ is modular for $A=M(1/2,-1/2,1)$ and the choices of $(B,C,v,L)$ in Table \ref{tab-12}.

\begin{table}[htbp]
\caption{Modular quadruples $(A,B,C,v+L)$ with $A=M(1/2,-1/2,1)$} \label{tab-12}
\centering
\renewcommand{\arraystretch}{1.5}
\begin{tabular}{|c|c|c|c|}
  \hline
   $B$ & $C$ & $v$ & $L$ \\ \hline
   $(1/2,-1/2)^\mathrm{T}$ & $1/21$  & $(0,0)^\mathrm{T}$, $(1,0)^\mathrm{T}$ & \multirow{3}{*}{$\mathbb{Z}(2,0)^\mathrm{T}+\mathbb{Z}(0,1)^\mathrm{T}$} \\ \cline{1-3}
   $(0,0)^\mathrm{T}$  & $-5/84$ & $(0,0)^\mathrm{T}$, $(1,0)^\mathrm{T}$ &  \\ \cline{1-3}
   $(1/2,0)^\mathrm{T}$ & $1/84$ & $(0,0)^\mathrm{T}$, $(1,0)^\mathrm{T}$ & \\ \hline
   $(1/2,-1/2)^\mathrm{T}$ & $1/21$ & $(0,0)^\mathrm{T}$, $(0,1)^\mathrm{T}$ & \multirow{3}{*}{$\mathbb{Z}(2,0)^\mathrm{T}+\mathbb{Z}(0,2)^\mathrm{T}$} \\ \cline{1-3}
    $(0,0)^\mathrm{T}$  &  $-5/84$ &  $(1,0)^\mathrm{T}$, $(1,1)^\mathrm{T}$ &  \\ \cline{1-3}
   $(1/2,0)^\mathrm{T}$ & $1/84$  &  $(1,0)^\mathrm{T}$, $(1,1)^\mathrm{T}$  & \\
  \hline
\end{tabular}
\end{table}

\begin{theorem}
    We have
    \begin{align}
        \sum_{i,j\ge0}\frac{q^{i^2-2ij+2j^2+i-j}}{(q;q)_{2i}(q;q)_{2j}}&=\frac{1}{(q;q^2)_\infty^2(q^{2},q^{3},q^{4},q^{10},q^{11},q^{12};q^{14})_\infty},\label{eq4-1}\\
        \sum_{i,j\ge0}\frac{q^{i^2-2ij+2j^2+j}}{(q;q)_{2i}(q;q)_{2j+1}}&=\frac{1}{(q;q^2)_\infty^2(q^{2},q^{3},q^{4},q^{10},q^{11},q^{12};q^{14})_\infty},\label{eq4-2}\\
        \sum_{i,j\ge0}\frac{q^{i^2-2ij+2j^2+i-j}}{(q;q)_{2i+1}(q;q)_{2j}}&=\frac{1}{(q;q^2)_\infty^2(q,q^{4},q^{6},q^{8},q^{10},q^{13};q^{14})_\infty},\label{eq4-3}\\
        \sum_{i,j\ge0}\frac{q^{i^2-2ij+2j^2+j}}{(q;q)_{2i+1}(q;q)_{2j+1}}&=\frac{1}{(q;q^2)_\infty^2(q,q^{4},q^{6},q^{8},q^{10},q^{13};q^{14})_\infty},\label{eq4-5}\\
        \sum_{i,j\ge0}\frac{q^{i^2-2ij+2j^2+2i-j}}{(q;q)_{2i+1}(q;q)_{2j}}&=\frac{1}{(q;q^2)_\infty^2(q^{2},q^{5},q^{6},q^{8},q^{9},q^{12};q^{14})_\infty},\label{eq4-4}\\
        \sum_{i,j\ge0}\frac{q^{i^2-2ij+2j^2+i+j}}{(q;q)_{2i+1}(q;q)_{2j+1}}&=\frac{1}{(q;q^2)_\infty^2(q^{2},q^{5},q^{6},q^{8},q^{9},q^{12};q^{14})_\infty}.\label{eq4-6}
    \end{align}
\end{theorem}
\begin{proof}
This can be proved using \eqref{ep3-1}, \eqref{ep3-2} and \eqref{ep3-3} in the same way as the proof of Theorem \ref{thm34121-sub}. We omit the details.
\end{proof}

\subsection{The case $(\alpha, 1-\alpha, \alpha)$} Modularity of the full Nahm sum $f_{A,B,C}(q)$ with
\begin{align}
A=M(\alpha, 1-\alpha, \alpha), ~~B=(\alpha\nu,-\alpha\nu)~ (\nu\in \mathbb{Q}), ~~
C=\alpha \nu^2/2-1/24
\end{align}
was proved by Zagier \cite[p.\ 46]{Zagier}. We find that the partial Nahm sum $f_{A,B,C,v+L}(q)$ is modular for the choices of $(A,B,C,v,L)$ in Table \ref{tab-13}.

\begin{table}[H]
\caption{Modular quadruples $(A,B,C,v+L)$  with $A=M(\alpha,1-\alpha,\alpha)$ ($\alpha>0$)} \label{tab-13}
\centering
\renewcommand{\arraystretch}{1.5}
\begin{tabular}{|c|c|c|c|}
  \hline
   $B$ & $C$ & $v$ & $L$ \\ \hline
  $(-1/2,-1/2)^\mathrm{T}$ & $1/(8\alpha)-1/24$  & $(1,0)^\mathrm{T}$, $(0,1)^\mathrm{T}$ & \multirow{2}{*}{$\mathbb{Z}(2,0)^\mathrm{T}+\mathbb{Z}(0,2)^\mathrm{T}$} \\ \cline{1-3}
   $(0,0)^\mathrm{T}$  & $-1/24$ & $(1,0)^\mathrm{T}$, $(0,1)^\mathrm{T}$ &
  \\ \hline
\end{tabular}
\end{table}

\begin{theorem}\label{thm-alpha}
For $\alpha>0$ we have
    \begin{align}
        \sum_{i,j\ge0}\frac{q^{2\alpha i^2+4(1-\alpha)ij+2\alpha j^2+(2{\alpha}-1)(i-j)}}{(q;q)_{2i+1}(q;q)_{2j}}&=\frac{(-q,-q^{4\alpha-1},q^{4\alpha};q^{4\alpha})_{\infty}}{(q;q)_\infty},\label{eq7-1}\\
        \sum_{i,j\ge0}\frac{q^{2 \alpha i^2+4(1-\alpha)ij+2{\alpha}j^2+2\alpha i+(2-2 \alpha)j}}{(q;q)_{2i+1}(q;q)_{2j}}&=\frac{(q^{8\alpha};q^{8\alpha})_\infty^2}{(q;q)_\infty(q^{4\alpha};q^{4\alpha})_\infty}.\label{eq7-2}
    \end{align}
\end{theorem}
\begin{proof}
    We define
    \begin{align}
        R(u,v)=R(u,v;q):=\sum_{i,j\ge0}\frac{q^{\frac{\alpha}{2}i^2+(1-\alpha)ij+\frac{\alpha}{2}j^2}u^iv^j}{(q;q)_{i}(q;q)_{j}}.
    \end{align}
Clearly we have
    \begin{align}
        R(u,v)=R(v,u).
    \end{align}

(1) By Lemma \ref{lem-mod2} we have
    \begin{align}
        &\sum_{i,j\ge0}\frac{q^{2\alpha i^2+(4-4\alpha)ij+2\alpha j^2+(2 \alpha -1)(i-j)}}{(q;q)_{2i+1}(q;q)_{2j}} \nonumber \\
        &=\frac{1}{4}q^{\frac{1}{2}-\frac{\alpha}{2}}\Big(R(q^{-\frac{1}{2}},q^{-\frac{1}{2}})   +R(q^{-\frac{1}{2}},-q^{-\frac{1}{2}})-R(-q^{-\frac{1}{2}},q^{-\frac{1}{2}})  -R(-q^{-\frac{1}{2}},-q^{-\frac{1}{2}})\Big) \nonumber \\
        &=\frac{1}{4}q^{\frac{1}{2}-\frac{\alpha}{2}}\Big(R(q^{-\frac{1}{2}},q^{-\frac{1}{2}})-R(-q^{-\frac{1}{2}},-q^{-\frac{1}{2}})\Big). \label{R-a-proof-1}
    \end{align}
Recall the following identity found by Cao and Wang \cite[Theorem 3.6]{Cao-Wang}:
    \begin{align}
   & \sum_{i,j\ge0}\frac{(-1)^{i+j}u^{i-j}q^{(i^2-i+j^2-j+(a-1)(i-j)^2)/2}}{(q;q)_{i}(q;q)_{j}} \label{cao-wang 2}\\
        &=\frac{(u^{-1}q^{(a-1)/2},uq^{(a+1)/2},q^{a};q^{a})_{\infty}+(uq^{(a-1)/2},u^{-1}q^{(a+1)/2},q^{a};q^{a})_{\infty}}{(q;q)_{\infty}}.\notag
    \end{align}
Setting $(u,a)=(-1,\alpha)$ and $(1,\alpha)$ in \eqref{cao-wang 2}, we obtain
\begin{align}
        &R(q^{-\frac{1}{2}},q^{-\frac{1}{2}})=\frac{2(-q^{\alpha/2+1/2},-q^{\alpha/2-1/2},q^{\alpha};q^{\alpha})_\infty}{(q;q)_\infty}, \label{R-a-1}\\
        &R(-q^{-\frac{1}{2}},-q^{-\frac{1}{2}})=\frac{2(q^{\alpha/2+1/2},q^{\alpha/2-1/2},q^{\alpha};q^{\alpha})_\infty}{(q;q)_\infty}. \label{R-a-2}
    \end{align}
Substituting \eqref{R-a-1} and \eqref{R-a-2} into \eqref{R-a-proof-1} and then using \eqref{jac odd} with $(a,b)=({\alpha}/{2},{1}/{2})$ to simplify, we obtain \eqref{eq7-1}.

(2) By Lemma \ref{lem-mod2} we have
    \begin{align}
        &\sum_{i,j\ge0}\frac{q^{{2\alpha}i^2+(4-4\alpha)ij+{2\alpha}j^2+2{\alpha}i+(2-2{\alpha})j}}{(q;q)_{2i+1}(q;q)_{2j}} \nonumber \\
        &=\frac{1}{4}q^{-\frac{\alpha}{2}}\Big(R(1,1)+R(1,-1)-R(-1,1)-R(-1,-1)\Big) \nonumber \\
        &=\frac{1}{4}q^{-\frac{\alpha}{2}}\Big(R(1,1)-R(-1,-1)\Big). \label{R-a-proof-2}
    \end{align}
Setting $(u,a)=(q^{1-\alpha/2},\alpha-1)$ and $(-q^{1-\alpha/2},\alpha-1)$ in \eqref{cao-wang 1},   We obtain
\begin{align}
        &R(1,1)=\frac{(-q^{\alpha/2},-q^{\alpha/2},q^{\alpha};q^{\alpha})_\infty}{(q;q)_\infty}, \label{R-a-3}\\
 &R(-1,-1)=\frac{(q^{\alpha/2},q^{\alpha/2},q^{\alpha};q^{\alpha})_\infty}{(q;q)_\infty}. \label{R-a-4}
\end{align}
Substituting \eqref{R-a-3} and \eqref{R-a-4} into \eqref{R-a-proof-2} and then using \eqref{jac odd} with $(a,b)=({\alpha}/{2},0)$ to simplify, we obtain \eqref{eq7-2}.
\end{proof}
\begin{rem}
   The identities \eqref{eq8-3} and \eqref{eq8-4} correspond to the special case $\alpha=1/2$ of Theorem \ref{thm-alpha}.
\end{rem}
\begin{rem}
We explain a little bit about the convergence of $R(u,v)$. When $0<\alpha\leq 1$, we have
\begin{align}
    \frac{\alpha}{2}i^2+(1-\alpha)ij+\frac{\alpha}{2}j^2\geq \frac{\alpha}{2}i^2+\frac{\alpha}{2}j^2, \quad \forall i,j\geq 0.
\end{align}
When $\alpha>1/2$, the matrix $A=M(\alpha,1-\alpha,\alpha)$ is positive definite. From Lemma \ref{lem-convergence} and Corollary \ref{cor-positive} we know that $R(u,v)$ is absolutely convergent when $|q|<1$. This also explains that \eqref{cao-wang 2} holds for any $a>0$, which was not clearly stated in \cite{Cao-Wang}.
\end{rem}

The special case $\alpha=3/4$ contains more modular quadruples as recorded in Table \ref{tab-13-special}.
\begin{table}[H]
\caption{Modular quadruples $(A,B,C,v+L)$ with $A=M(3/4,1/4,3/4)$} \label{tab-13-special}
\centering
\renewcommand{\arraystretch}{1.5}
\begin{tabular}{|c|c|c|c|}
  \hline
   $B$ & $C$ & $v$ & $L$ \\ \hline
 $(-1/4,1/4)^\mathrm{T}$ & $0$  & $(0,0)^\mathrm{T}$, $(1,1)^\mathrm{T}$ & \multirow{6}{*}{$\mathbb{Z}(2,0)^\mathrm{T}+\mathbb{Z}(0,2)^\mathrm{T}$}  \\ \cline{1-3}
  $(1/4,-1/4)^\mathrm{T}$ & $0$  & $(0,0)^\mathrm{T}$, $(1,1)^\mathrm{T}$ &  \\ \cline{1-3}
  $(-1/2,-1/2)^\mathrm{T}$ & $ 1/8$ & $(1,0)^\mathrm{T}$, $(0,1)^\mathrm{T}$ &   \\ \cline{1-3}
   $(-1/2,1/2)^\mathrm{T}$ & $1/8$ & $(1,0)^\mathrm{T}$, $(0,1)^\mathrm{T}$ &  \\ \cline{1-3}
   $(1/2,-1/2)^\mathrm{T}$ & $1/8$ & $(1,0)^\mathrm{T}$, $(0,1)^\mathrm{T}$ &  \\ \cline{1-3}
  $(0,0)^\mathrm{T}$ & $-1/24$ & $(1,0)^\mathrm{T}$, $(0,1)^\mathrm{T}$ &
  \\ \hline
\end{tabular}
\end{table}

\begin{theorem}\label{thm341434}
    We have
    \begin{align}
        \sum_{i,j\ge0}\frac{q^{\frac{3}{2}i^2+ij+\frac{3}{2}j^2-\frac{1}{2}i+\frac{1}{2}j}}{(q;q)_{2i}(q;q)_{2j}}&=\frac{(-q^{5},-q^{7},q^{12};q^{12})_\infty}{(q;q)_\infty},\label{eq9-1}\\
        \sum_{i,j\ge0}\frac{q^{\frac{3}{2}i^2+ij+\frac{3}{2}j^2+\frac{3}{2}i+\frac{5}{2}j}}{(q;q)_{2i+1}(q;q)_{2j+1}}&=\frac{(-q,-q^{11},q^{12};q^{12})_\infty}{(q;q)_\infty}, \label{eq9-6} \\
        \sum_{i,j\ge0}\frac{q^{\frac{3}{2}i^2+ij+\frac{3}{2}j^2+\frac{1}{2}i-\frac{1}{2}j}}{(q;q)_{2i+1}(q;q)_{2j}}&=\frac{(q^3;q^3)_\infty}{(q;q)_\infty(q,q^5;q^6)_\infty},\label{eq9-2}\\
        \sum_{i,j\ge0}\frac{q^{\frac{3}{2}i^2+ij+\frac{3}{2}j^2+\frac{1}{2}i+\frac{3}{2}j}}{(q;q)_{2i+1}(q;q)_{2j}}&=\frac{(-q^{5},-q^{7},q^{12};q^{12})_\infty}{(q;q)_\infty},\label{eq9-3}\\
         \sum_{i,j\ge0}\frac{q^{\frac{3}{2}i^2+ij+\frac{3}{2}j^2+\frac{5}{2}i-\frac{1}{2}j}}{(q;q)_{2i+1}(q;q)_{2j}}&=\frac{(-q,-q^{11},q^{12};q^{12})_\infty}{(q;q)_\infty}, \label{eq9-5} \\
        \sum_{i,j\ge0}\frac{q^{\frac{3}{2}i^2+ij+\frac{3}{2}j^2+\frac{3}{2}i+\frac{1}{2}j}}{(q;q)_{2i+1}(q;q)_{2j}}&=\frac{(q^6;q^6)_\infty}{(q;q)_\infty(q^3;q^6)_\infty}.\label{eq9-4}
    \end{align}
\end{theorem}

Before giving a proof, we need some preparations.
\begin{lemma}\label{lem-zero}
Let $t\geq 0$ be an integer  and $s$ an odd integer with $-1\leq s\leq 2t+1$. For any $n\ge 1$ we have
    \begin{align}
        \sum_{i=-n-t}^{n}\frac{(-1)^{i}q^{(i^2+si)/2}}{(q;q)_{n-i}(q;q)_{n+t+i}}=0.\label{bt-1}
    \end{align}
\end{lemma}
\begin{proof}
Letting $k=n-i$, using \eqref{euler-finite} with $(z,n)$ replaced by $(-q^{-n+(1-s)/2},2n+t)$, we have
    \begin{align*}
        &\sum_{i=-n-t}^{n}\frac{(-1)^{i}q^{(i^2+si)/2}}{(q;q)_{n-i}(q;q)_{n+t+i}}= \frac{(-1)^{n}}{(q;q)_{2n+t}}\sum_{k=0}^{2n+t}\frac{(-1)^{k}q^{((n-k)^2+s(n-k))/2} (q;q)_{2n+t}}{(q;q)_{k}(q;q)_{2n+t-k}}\\
        &=\frac{(-1)^{n}q^{(n^2+sn)/2}}{(q;q)_{2n+t}}\sum_{k=0}^{2n+t}q^{(k^2-k)/2}(-q^{-n+(1-s)/2})^{k}{   2n+t \brack k}\\
        &=\frac{(-1)^{n}q^{(n^2+sn)/2}}{(q;q)_{2n+t}}(q^{-n+(1-s)/2};q)_{2n+t}=0. \qedhere
    \end{align*}
\end{proof}

\begin{lemma}\label{lem-BP-example}
The following $(\alpha_n^{(i)}(a;q),\beta_n^{(i)}(a;q))$ ($i=1,2,3,4$, $a\in \{1,q\}$) form Bailey pairs relative to $a$:
\begin{align}
        \alpha_n^{(1)}(1;q) &= \begin{cases}
        1,  & n=0 \\
        0,  & n=2k+1,k\ge 0\\
        q^{2k^2}(q^{k}+q^{-k}), & n=2k,k\ge 1
        \end{cases}, \\
        \beta _n^{(1)}(1;q) &= \sum_{i+j=n}^{}\frac{q^{\frac{1}{2}(i-j)^2-\frac{1}{2}(i-j)}}{(q;q)_{2i}(q;q)_{2j}}; \\
        \alpha_n^{(2)}(q;q) &= \begin{cases}
        q^{2k^2-k},  & n=2k,k\ge 0 \\
        q^{2k^2+5k+3},  & n=2k+1,k\ge 0
        \end{cases}, \\
        \beta _n^{(2)}(q;q) &= (1-q)\sum_{i+j=n}^{}\frac{q^{\frac{1}{2}(i-j)^2-\frac{1}{2}(i-j)}}{(q;q)_{2i+1}(q;q)_{2j}}; \\
        \alpha_n^{(3)}(q;q) &= \begin{cases}
        q^{2k^2+3k},   &  n=2k,k\ge 0 \\
        q^{2k^2+k-1},  &  n=2k+1,k\ge 0
        \end{cases}, \\
        \beta _n^{(3)}(q;q) &= (1-q)\sum_{i+j=n}^{}\frac{q^{\frac{1}{2}(i-j)^2+\frac{3}{2}(i-j)}}{(q;q)_{2i+1}(q;q)_{2j}}; \\
        \alpha_n^{(4)}(1;q) &= \begin{cases}
        0,  & n=2k,k\ge 0 \\
        q^{2k^2+2k+\frac{1}{2}}(q^{k+\frac{1}{2}}+q^{-k-\frac{1}{2}}),  & n=2k+1,k\ge 0
        \end{cases} ,\\
        \beta _n^{(4)}(1;q) &= \sum_{i+j=n-1}^{}\frac{q^{\frac{1}{2}(i-j)^2-\frac{1}{2}(i-j)}}{(q;q)_{2i+1}(q;q)_{2j+1}} ~~(n \geq 1), \quad \beta _0^{(4)}(1;q)=0.
\end{align}
\end{lemma}

\begin{proof}
Let $t,s$ be integers satisfying the conditions in Lemma \ref{lem-zero}. We have
    \begin{align*}
        &\sum_{i+j=n}\frac{q^{\frac{1}{2}(i-j)^2+\frac{s}{2}(i-j)}}{(q;q)_{2i+t}(q;q)_{2j}}=\sum_{i=\lceil -t/2 \rceil}^{n}\frac{q^{\frac{1}{2}(2i-n)^2+\frac{s}{2}(2i-n)}}{(q;q)_{2i+t}(q;q)_{2n-2i}}\\
        &=\frac{1}{2}\sum_{i=-t}^{2n}\Big(\frac{q^{\frac{1}{2}(i-n)^2+\frac{s}{2}(i-n)}}{(q;q)_{i+t}(q;q)_{2n-i}}+\frac{(-1)^{i}q^{\frac{1}{2}(i-n)^2+\frac{s}{2}(i-n)}}{(q;q)_{i+t}(q;q)_{2n-i}}\Big)\\
        &=\frac{1}{2}\Big(\sum_{i=-n-t}^{n}\frac{q^{(i^2+si)/2}}{(q;q)_{n+i+t}(q;q)_{n-i}}+(-1)^{n}\sum_{i=-n-t}^{n}\frac{(-1)^{i}q^{(i^2+si)/2}}{(q;q)_{n+i+t}(q;q)_{n-i}}\Big)\\
        &=\frac{1}{2}\Big(\sum_{i=-n-t}^{n}\frac{q^{(i^2+si)/2}}{(q;q)_{n+i+t}(q;q)_{n-i}}+\sum_{i=-n-t}^{n}\frac{(-1)^{i}q^{(i^2+si)/2}}{(q;q)_{n+i+t}(q;q)_{n-i}}\Big) \quad \text{(by \eqref{bt-1})}\\
        &=\frac{1}{2}\Big(\sum_{i=0}^{n}\frac{q^{(i^2+si)/2}}{(q;q)_{n+i+t}(q;q)_{n-i}}+\sum_{i=-n-t}^{-1}\frac{q^{(i^2+si)/2}}{(q;q)_{n+i+t}(q;q)_{n-i}}\\
        &\quad\quad +\sum_{i=0}^{n}\frac{(-1)^{i}q^{(i^2+si)/2}}{(q;q)_{n+i+t}(q;q)_{n-i}}+\sum_{i=-n-t}^{-1}\frac{(-1)^{i}q^{(i^2+si)/2}}{(q;q)_{n+i+t}(q;q)_{n-i}}\Big)\\
        &=\frac{1}{2}\Big(\sum_{i=0}^{n}\frac{q^{(i^2+si)/2}}{(q;q)_{n+i+t}(q;q)_{n-i}}+\sum_{i=1-t}^{n}\frac{q^{(i^2+(2t-s)i+(t^2-st))/2}}{(q;q)_{n-i}(q;q)_{n+i+t}}\\
        &\quad\quad+\sum_{i=0}^{n}\frac{(-1)^{i}q^{(i^2+si)/2}}{(q;q)_{n+i+t}(q;q)_{n-i}}+\sum_{i=1-t}^{n}\frac{(-1)^{i+t}q^{(i^2+(2t-s)i+t^2-st)/2}}{(q;q)_{n-i}(q;q)_{n+1+i}}\Big) \nonumber \\
        &=\sum_{\begin{smallmatrix} i=0 \\ i \equiv 0 \!\!\!\! \pmod{2} \end{smallmatrix}}^{n}\frac{q^{(i^2+si)/2}}{(q;q)_{n+i+t}(q;q)_{n-i}}+\sum_{\begin{smallmatrix}i=1-t \\ i\equiv t \!\!\!\! \pmod{2} \end{smallmatrix}}^{n}\frac{q^{(i^2+(2t-s)i+(t^2-st))/2}}{(q;q)_{n-i}(q;q)_{n+i+t}}.
    \end{align*}
When $(t,s)=(0,-1)$, $(1-1)$ and $(1,3)$, we conclude that $(\alpha_n^{(i)},\beta_n^{(i)})$ is a Bailey pair for $i=1,2,3$, respectively.

Similarly, we have
    \begin{align*}
        &\sum_{i+j=n-1}^{}\frac{q^{\frac{1}{2}(i-j)^2-\frac{1}{2}(i-j)}}{(q;q)_{2i+1}(q;q)_{2j+1}}=\sum_{i=0}^{n-1}\frac{q^{\frac{1}{2}(2i-n+1)^2-\frac{1}{2}(2i-n+1)}}{(q;q)_{2i+1}(q;q)_{2n-2i-1}}\\
        &=\frac{1}{2}\sum_{k=0}^{2n}\Big(\frac{q^{\frac{1}{2}(k-n)^2-\frac{1}{2}(k-n)}}{(q;q)_{k}(q;q)_{2n-k}}-\frac{(-1)^{k}q^{\frac{1}{2}(k-n)^2-\frac{1}{2}(k-n)}}{(q;q)_{k}(q;q)_{2n-k}}\Big)\\
        &=\frac{1}{2}\Big(\sum_{k=-n}^{n}\frac{q^{(k^2-k)/2}}{(q;q)_{n-k}(q;q)_{n+k}}-(-1)^{n}\sum_{k=-n}^{n}\frac{(-1)^{k}q^{(k^2-k)/2}}{(q;q)_{n-k}(q;q)_{n+k}}\Big)\\
        &=\frac{1}{2}\Big(\sum_{k=-n}^{n}\frac{q^{(k^2-k)/2}}{(q;q)_{n-k}(q;q)_{n+k}}-\sum_{k=-n}^{n}\frac{(-1)^{k}q^{(k^2-k)/2}}{(q;q)_{n-k}(q;q)_{n+k}}\Big) \\
        &=\sum_{\begin{smallmatrix} k=-n \\ k\equiv 1 \!\!\!\! \pmod{2} \end{smallmatrix}}^{n}\frac{q^{(k^2-k)/2}}{(q;q)_{n-k}(q;q)_{n+k}}.
    \end{align*}
Here for the last equality we used \eqref{bt-1} with $(t,s)=(0,-1)$.
    This proves that $(\alpha_n^{(4)}(1;q),\beta_n^{(4)}(1;q))$ is a Bailey pair.
\end{proof}

\begin{proof}[Proof of Theorem \ref{thm341434}]
Note that \eqref{eq9-2} and \eqref{eq9-4} correspond to the case $\alpha=3/4$ of Theorem \ref{thm-alpha}. To prove the remaining identities, we define
\begin{align}\label{F-st-start}
 &F(t,s)=F(t,s;q):=\sum_{i,j\ge0}\frac{q^{(i+j)^2+\frac{1}{2}(i-j)^2+t(i+j)+\frac{s}{2}(i-j)}}{(q;q)_{2i+t}(q;q)_{2j}} \nonumber \\
 &=\sum_{n=0}^{\infty}q^{n^2+tn}\sum_{i+j=n}\frac{q^{\frac{1}{2}(i-j)^2+\frac{s}{2}(i-j)}}{(q;q)_{2i+t}(q;q)_{2j}}.
\end{align}
From Lemmas \ref{lem-BP-id} and \ref{lem-BP-example} and \eqref{F-st-start} we have
    \begin{align}
       &F(0,-1)=\sum_{n=0}^{\infty}q^{n^2}\beta_n^{(1)}(1;q)=\frac{1}{(q;q)_\infty}\sum_{n=0}^{\infty}q^{n^2}\alpha_n^{(1)}(1;q) \nonumber \\
       &=\frac{1}{(q;q)_\infty}\sum_{n=-\infty}^{\infty}q^{6n^2+n}, \label{proof-ts-1} \\
        &F(1,-1)=\frac{1}{1-q}\sum_{n=0}^{\infty}q^{n^2+n}\beta_{n}^{(2)}(q;q)=\frac{1}{(1-q)(q^2;q)_\infty}\sum_{n=0}^{\infty}q^{n^2+n}\alpha_{n}^{(2)}(q;q) \nonumber \\
        &=\frac{1}{(q;q)_\infty}\Big(\sum_{n=0}^{\infty}q^{6n^2+n}+\sum_{n=0}^{\infty}q^{6n^2+11n+5}\Big) =\frac{1}{(q;q)_\infty}\sum_{n=-\infty}^{\infty}q^{6n^2+n}, \label{proof-ts-2} \\
       &F(1,3)=\frac{1}{1-q}\sum_{n=0}^{\infty}q^{n^2+n}\beta_{n}^{(3)}(q;q)=\frac{1}{(1-q)(q^2;q)_\infty}\sum_{n=0}^{\infty}q^{n^2+n}\alpha_{n}^{(3)}(q;q) \nonumber \\
        &=\frac{1}{(q;q)_\infty}\Big(\sum_{n=0}^{\infty}q^{6n^2+5n}+\sum_{n=0}^{\infty}q^{6n^2+7n+1}\Big)=\frac{1}{(q;q)_\infty}\sum_{n=-\infty}^{\infty}q^{6n^2+5n}. \label{proof-ts-3}
    \end{align}
Applying \eqref{Jacobi} to the right side of \eqref{proof-ts-1}--\eqref{proof-ts-3}, we obtain \eqref{eq9-1}, \eqref{eq9-3} and \eqref{eq9-5}, respectively.

   Similarly, we have
    \begin{align*}
        &\sum_{i,j\ge0}\frac{q^{\frac{3}{2}i^2+ij+\frac{3}{2}j^2+\frac{3}{2}i+\frac{5}{2}j}}{(q;q)_{2i+1}(q;q)_{2j+1}}=q^{-1}\sum_{n=0}^{\infty}q^{n^2+2n+1}\sum_{i+j=n}^{}\frac{q^{\frac{1}{2}(i-j)^2-\frac{1}{2}(i-j)}}{(q;q)_{2i+1}(q;q)_{2n-2i+1}}\\
        &=q^{-1}\sum_{n=0}^{\infty}q^{n^2}\beta_{n}^{(4)}(1;q)=\frac{1}{q(q;q)_\infty}\sum_{n=0}^{\infty}q^{n^2}\alpha_n^{(4)}(1;q)\\
        &=\frac{1}{(q;q)_\infty}\sum_{n=0}^{\infty}(q^{6n^2+7n+1}+q^{6n^2+5n})=\frac{1}{(q;q)_\infty}\sum_{n=-\infty}^{\infty}q^{6n^2+5n}.
    \end{align*}
Again by \eqref{Jacobi} we obtain \eqref{eq9-6}.
\end{proof}

\subsection{The case $(3/4, -1/4, 3/4)$}
The identities we discovered are sated in Theorem \ref{thm-last}.  To give a proof, we first establish a useful lemma.
\begin{lemma}
For any integer $n\geq 0$ we have
\begin{align}
       \frac{q^{\frac{1}{2}n^2}}{(q;q)_{n}}&=\sum_{i+j=n}\frac{q^{i^2+j^2-\frac{1}{2}i+\frac{1}{2}j}}{(q^2;q^2)_{i}(q^2;q^2)_{j}}, \label{sum-1} \\
        \frac{(-1)^{\frac{1}{2}n^2-\frac{1}{2}n}q^{\frac{1}{2}n^2}}{(-q;-q)_{n}}&=\sum_{i+j=n}\frac{(-1)^{j}q^{i^2+j^2-\frac{1}{2}i+\frac{1}{2}j}}{(q^2;q^2)_{i}(q^2;q^2)_{j}}, \label{sum-2} \\
        \frac{(-1)^{\frac{1}{2}n^2+\frac{1}{2}n}q^{\frac{1}{2}n^2}}{(-q;-q)_{n}}&=\sum_{i+j=n}\frac{(-1)^{i}q^{i^2+j^2-\frac{1}{2}i+\frac{1}{2}j}}{(q^2;q^2)_{i}(q^2;q^2)_{j}}. \label{sum-3}
 \end{align}
\end{lemma}
\begin{proof}
Note that
    \begin{align}
       & \sum_{n=0}^{\infty}\frac{q^{\frac{1}{2}n^2}z^n}{(q;q)_{n}}=(-zq^{\frac{1}{2}};q)_{\infty}=(-zq^{\frac{1}{2}};q^2)_{\infty}(-zq^{\frac{3}{2}};q^2)_{\infty} \nonumber \\
       &=\sum_{i=0}^{\infty}\frac{q^{i^2-\frac{1}{2}i}z^i}{(q^2;q^2)_{i}}\sum_{j=0}^{\infty}\frac{q^{j^2+\frac{1}{2}j}z^j}{(q^2;q^2)_{j}}, \label{410-2} \\
 &\sum_{n=0}^{\infty}\frac{(-1)^{\frac{1}{2}n^2-\frac{1}{2}n}q^{\frac{1}{2}n^2}z^n}{(-q;-q)_{n}}=(-zq^{\frac{1}{2}};-q)_{\infty}=(-zq^{\frac{1}{2}};q^2)_{\infty}(zq^{\frac{3}{2}};q^2)_{\infty}\notag\\
        &=\sum_{i=0}^{\infty}\frac{q^{i^2-\frac{1}{2}i}z^i}{(q^2;q^2)_{i}}\sum_{j=0}^{\infty}\frac{(-1)^jq^{j^2+\frac{1}{2}j}z^j}{(q^2;q^2)_{j}}. \label{mid-410}
\end{align}
Comparing the coefficients of $z^n$ on both sides of \eqref{410-2} and \eqref{mid-410}, we obtain \eqref{sum-1} and \eqref{sum-2}, respectively. Multiplying both sides of \eqref{sum-2} by $(-1)^n$, we obtain \eqref{sum-3}.
\end{proof}
We refer the reader to \cite[p.\ 638]{VZ} for a generalization of \eqref{sum-1}.

\begin{proof}[Proof of Theorem \ref{thm-last}]
    We define
    \begin{align}
        R(x,y;q):=\sum_{i,j\ge0}\frac{q^{\frac{3}{8}i^2-\frac{1}{4}ij+\frac{3}{8}j^2}x^iy^j}{(q;q)_i(q;q)_j}.
    \end{align}
Clearly we have
\begin{align}\label{R-symmetric}
R(x,y;q)=R(y,x;q).
\end{align}
(1)  By Lemma \ref{lem-mod2} we have
    \begin{align}
        &\sum_{i,j\ge0}\frac{q^{\frac{3}{2}i^2-ij+\frac{3}{2}j^2-\frac{1}{2}i+\frac{1}{2}j}}{(q;q)_{2i}(q;q)_{2j}}=\frac{1}{4}\Big(R(q^{-\frac{1}{4}},q^{\frac{1}{4}};q)+R(-q^{-\frac{1}{4}},-q^{\frac{1}{4}};q) \nonumber \\
        &\quad+R(q^{-\frac{1}{4}},-q^{\frac{1}{4}};q)+R(-q^{-\frac{1}{4}},q^{\frac{1}{4}};q)\Big),  \label{eq-last-1} \\
        &\sum_{i,j\ge0}\frac{q^{\frac{3}{2}i^2-ij+\frac{3}{2}j^2+\frac{1}{2}i+\frac{3}{2}j}}{(q;q)_{2i+1}(q;q)_{2j+1}}=\frac{1}{4q^{1/2}}\Big(R(q^{-\frac{1}{4}},q^{\frac{1}{4}};q)+R(-q^{-\frac{1}{4}},-q^{\frac{1}{4}};q) \nonumber \\
        &\quad-R(q^{-\frac{1}{4}},-q^{\frac{1}{4}};q)-R(-q^{-\frac{1}{4}},q^{\frac{1}{4}};q)\Big). \label{eq-last-2}
    \end{align}

By \eqref{sum-1} we have
    \begin{align}
        &R(q^{-\frac{1}{2}},q^{\frac{1}{2}};q^2)=\sum_{n=0}^{\infty}q^{-\frac{1}{4}n^2}\sum_{i+j=n}\frac{q^{i^2+j^2-\frac{1}{2}i+\frac{1}{2}j}}{(q^2;q^2)_{i}(q^2;q^2)_{j}}=\sum_{n=0}^{\infty}\frac{q^{\frac{1}{4}n^2}}{(q;q)_{n}}, \label{410-3}  \\
        &R(-q^{-\frac{1}{2}},-q^{\frac{1}{2}};q^2)=\sum_{n=0}^{\infty}(-1)^{n}q^{-\frac{1}{4}n^2}\sum_{i+j=n}\frac{q^{i^2+j^2-\frac{1}{2}i+\frac{1}{2}j}}{(q^2;q^2)_{i}(q^2;q^2)_{j}}=\sum_{n=0}^{\infty}\frac{(-1)^{n}q^{\frac{1}{4}n^2}}{(q;q)_{n}}. \label{410-4}
    \end{align}
Adding \eqref{410-3} and \eqref{410-4} up, and then using \eqref{s-79}, we deduce that
\begin{align}\label{last-R-1}
        R(q^{-\frac{1}{2}},q^{\frac{1}{2}};q^2)+R(-q^{-\frac{1}{2}},-q^{\frac{1}{2}};q^2)&=2\sum_{n=0}^{\infty}\frac{q^{n^2}}{(q;q)_{2n}}=\frac{2}{(q;q^2)_\infty(q^{4},q^{16};q^{20})_{\infty}}.
\end{align}

By \eqref{sum-2} and \eqref{sum-3} we have
    \begin{align}
        &R(q^{-\frac{1}{2}},-q^{\frac{1}{2}};q^2)=\sum_{n=0}^{\infty}q^{-\frac{1}{4}n^2}\sum_{i+j=n}\frac{(-1)^{j}q^{i^2+j^2-\frac{1}{2}i+\frac{1}{2}j}}{(q^2;q^2)_{i}(q^2;q^2)_{j}}=\sum_{n=0}^{\infty}\frac{(-1)^{\frac{1}{2}n^2-\frac{1}{2}n}q^{\frac{1}{4}n^2}}{(-q;-q)_{n}}, \label{add-R-1} \\
        &R(-q^{-\frac{1}{2}},q^{\frac{1}{2}};q^2)=\sum_{n=0}^{\infty}q^{-\frac{1}{4}n^2}\sum_{i+j=n}\frac{(-1)^{i}q^{i^2+j^2-\frac{1}{2}i+\frac{1}{2}j}}{(q^2;q^2)_{i}(q^2;q^2)_{j}}=\sum_{n=0}^{\infty}\frac{(-1)^{\frac{1}{2}n^2+\frac{1}{2}n}q^{\frac{1}{4}n^2}}{(-q;-q)_{n}}. \label{add-R-2}
    \end{align}
Adding \eqref{add-R-1} and \eqref{add-R-2} up, and then using \eqref{s-79} with $q$ replaced by $-q$, we deduce that
    \begin{align}\label{last-R-2}
        R(q^{-\frac{1}{2}},-q^{\frac{1}{2}};q^2)+R(-q^{-\frac{1}{2}},q^{\frac{1}{2}};q^2)&=2\sum_{n=0}^{\infty}\frac{(-1)^{n}q^{n^2}}{(-q;-q)_{2n}}=\frac{2}{(-q;q^2)_\infty(q^{4},q^{16};q^{20})_{\infty}}.
    \end{align}
Substituting \eqref{last-R-1} and \eqref{last-R-2} with $q$ replaced by $q^{1/2}$ into \eqref{eq-last-1} and \eqref{eq-last-2}, and then using \eqref{ss-1} and \eqref{ss-2} to simplify, we obtain \eqref{eq10-1} and \eqref{eq10-4}.

(2)   By Lemma \ref{lem-mod2}  we have
\begin{align}
        &\sum_{i,j\ge0}\frac{q^{\frac{3}{2}i^2-ij+\frac{3}{2}j^2+\frac{3}{2}i+\frac{1}{2}j}}{(q;q)_{2i+1}(q;q)_{2j}}  \label{eq-last-3} \\
        &=\frac{1}{4}q^{-\frac{3}{8}}\Big(R(1,q^{\frac{1}{2}};q)-R(-1,-q^{\frac{1}{2}};q)+R(1,-q^{\frac{1}{2}};q)-R(-1,q^{\frac{1}{2}};q)\Big), \nonumber   \\
       & \sum_{i,j\ge0}\frac{q^{\frac{3}{2}i^2-ij+\frac{3}{2}j^2-\frac{1}{2}i+\frac{5}{2}j}}{(q;q)_{2i}(q;q)_{2j+1}} \label{eq-last-4} \\
       &=\frac{1}{4}q^{-\frac{7}{8}}\Big(R(1,q^{\frac{1}{2}};q)-R(-1,-q^{\frac{1}{2}};q)-R(1,-q^{\frac{1}{2}},q)+R(-1,q^{\frac{1}{2}};q)\Big).  \nonumber
\end{align}
Using \eqref{sum-1}--\eqref{sum-3} we have
    \begin{align}
        &R(1,q;q^2)=\sum_{n=0}^{\infty}q^{-\frac{1}{4}n^2+\frac{1}{2}n}\sum_{i+j=n}\frac{q^{i^2+j^2-\frac{1}{2}i+\frac{1}{2}j}}{(q^2;q^2)_{i}(q^2;q^2)_{j}}=\sum_{n=0}^{\infty}\frac{q^{\frac{1}{4}n^2+\frac{1}{2}n}}{(q;q)_{n}}, \label{proof-R-5}\\
        &R(-1,-q;q^2)=\sum_{n=0}^{\infty}(-1)^nq^{-\frac{1}{4}n^2+\frac{1}{2}n}\sum_{i+j=n}\frac{q^{i^2+j^2-\frac{1}{2}i+\frac{1}{2}j}}{(q^2;q^2)_{i}(q^2;q^2)_{j}}=\sum_{n=0}^{\infty}\frac{(-1)^{n}q^{\frac{1}{4}n^2+\frac{1}{2}n}}{(q;q)_{n}},\label{proof-R-6} \\
        &R(1,-q;q^2)=\sum_{n=0}^{\infty}q^{-\frac{1}{4}n^2+\frac{1}{2}n}\sum_{i+j=n}\frac{(-1)^{j}q^{i^2+j^2-\frac{1}{2}i+\frac{1}{2}j}}{(q^2;q^2)_{i}(q^2;q^2)_{j}}=\sum_{n=0}^{\infty}\frac{(-1)^{\frac{1}{2}n^2-\frac{1}{2}n}q^{\frac{1}{4}n^2+\frac{1}{2}n}}{(-q;-q)_{n}}, \label{proof-R-7} \\
        &R(-1,q;q^2)=\sum_{n=0}^{\infty}q^{-\frac{1}{4}n^2+\frac{1}{2}n}\sum_{i+j=n}\frac{(-1)^{i}q^{i^2+j^2-\frac{1}{2}i+\frac{1}{2}j}}{(q^2;q^2)_{i}(q^2;q^2)_{j}}=\sum_{n=0}^{\infty}\frac{(-1)^{\frac{1}{2}n^2+\frac{1}{2}n}q^{\frac{1}{4}n^2+\frac{1}{2}n}}{(-q;-q)_{n}}. \label{proof-R-8}
    \end{align}
From the above identities and \eqref{s-96}, we deduce that
    \begin{align}
        R(1,q;q^2)-R(-1,-q;q^2)&=2q^{3/4}\sum_{n=0}^{\infty}\frac{q^{n^2+2n}}{(q;q)_{2n+1}}=\frac{2q^{3/4}}{(q;q^2)_\infty(q^{8},q^{12};q^{20})_{\infty}}, \label{last-use-3} \\
        R(1,-q;q^2)-R(-1,q;q^2)&=2q^{3/4}\sum_{n=0}^{\infty}\frac{(-1)^{n}q^{n^2+2n}}{(-q;-q)_{2n+1}}=\frac{2q^{3/4}}{(-q;q^2)_\infty(q^{8},q^{12};q^{20})_{\infty}}. \label{last-use-4}
    \end{align}
Substituting \eqref{last-use-3} and \eqref{last-use-4} with $q$ replaced by $q^{1/2}$ into \eqref{eq-last-3} and \eqref{eq-last-4}, and then using \eqref{ss-1} and \eqref{ss-2} to simplify, we obtain \eqref{eq10-2} and \eqref{eq10-3}.
\end{proof}
Alert readers may have noted that the lattice cosets in Table \ref{tab-14} is incomplete. Here we discuss the remaining cases.
By Lemma \ref{lem-mod2} the partial Nahm sums $f_{A,B,C,v+L}(q)$ with $B=(-1/4,1/4)^\mathrm{T}$ (or $(1/4,-1/4)^\mathrm{T}$) and $v=(1,0)^\mathrm{T}$ or $(0,1)^\mathrm{T}$ can be calculated as
    \begin{align}
        &\sum_{i,j\ge0}\frac{q^{\frac{3}{2}i^2-ij+\frac{3}{2}j^2+i}}{(q;q)_{2i+1}(q;q)_{2j}}=\frac{1}{4}q^{-1/8}\Big(R(q^{-\frac{1}{4}},q^{\frac{1}{4}};q)+R(q^{-\frac{1}{4}},-q^{\frac{1}{4}};q) \nonumber \\
        &\quad -R(-q^{-\frac{1}{4}},q^{\frac{1}{4}};q)-R(-q^{-\frac{1}{4}},-q^{\frac{1}{4}};q)\Big),  \label{add-eq-last-1} \\
        &\sum_{i,j\ge0}\frac{q^{\frac{3}{2}i^2-ij+\frac{3}{2}j^2-i+2j}}{(q;q)_{2i}(q;q)_{2j+1}}=\frac{1}{4}q^{-5/8}\Big(R(q^{-\frac{1}{4}},q^{\frac{1}{4}};q)-R(q^{-\frac{1}{4}},-q^{\frac{1}{4}};q) \nonumber \\
        &\quad +R(-q^{-\frac{1}{4}},q^{\frac{1}{4}};q)-R(-q^{-\frac{1}{4}},-q^{\frac{1}{4}};q)\Big). \label{add-eq-last-2}
    \end{align}
Subtracting \eqref{410-4} from \eqref{410-3} and using \eqref{s-94}, we deduce that
\begin{align}\label{add-last-R-1}
        &R(q^{-\frac{1}{2}},q^{\frac{1}{2}};q^2)-R(-q^{-\frac{1}{2}},-q^{\frac{1}{2}};q^2)=2q^{\frac{1}{4}}\sum_{n=0}^{\infty}\frac{q^{n^2+n}}{(q;q)_{2n+1}} \nonumber \\
        &=2q^{\frac{1}{4}}\frac{(q^3,q^7,q^{10};q^{10})_\infty (q^{4},q^{16};q^{20})_\infty}{(q;q)_\infty}.
\end{align}
Subtracting \eqref{add-R-2} from \eqref{add-R-1}, we obtain
\begin{align}\label{add-last-R-2}
        R(q^{-\frac{1}{2}},-q^{\frac{1}{2}};q^2)-R(-q^{-\frac{1}{2}},q^{\frac{1}{2}};q^2)&=2q^{\frac{1}{4}}\sum_{n=0}^{\infty}\frac{(-1)^{n}q^{n^2+n}}{(-q;-q)_{2n+1}}.
\end{align}
We cannot find modular product expression for the sum on the right side. Therefore, we are not sure whether the partial Nahm sums $f_{A,B,C,v+L}(q)$ with $B\in \{(-1/4,1/4)^\mathrm{T}, (1/4,-1/4)^\mathrm{T}\}$ and $v\in \{(1,0)^\mathrm{T},(0,1)^\mathrm{T}\}$ are modular or not.

Similarly, from \eqref{proof-R-5}--\eqref{proof-R-8} and \eqref{s-99} we deduce that
\begin{align}
   & R(1,q;q^2)+R(-1,-q;q^2)=2\sum_{n=0}^\infty \frac{q^{n^2+n}}{(q;q)_{2n}}=2\frac{(q,q^9,q^{10};q^{10})_\infty (q^8,q^{12};q^{20})_\infty}{(q;q)_\infty}, \label{add-R-last-3} \\
   & R(1,-q;q^2)+R(-1,q;q^2)=2\sum_{n=0}^\infty \frac{(-1)^nq^{n^2+n}}{(-q;-q)_{2n}}. \label{add-R-last-4}
\end{align}
Since the modularity of the sum in the right side of \eqref{add-R-last-4} is not clear, we are not aware of the modularity of the partial Nahm sums $f_{A,B,C,v+L}(q)$ with $B\in \{(0,1/2)^\mathrm{T}, (1/2,0)^\mathrm{T}\}$ and $v\in \{(0,0)^\mathrm{T}, (1,1)^\mathrm{T}\}$. Similar phenomenon exists for the partial Nahm sums missed from Tables \ref{tab-13} and \ref{tab-13-special}.

\section{Concluding Remarks}\label{sec-remarks}
To summarize, we have introduced Nahm sums on lattice cosets and present 14 sets of rank two modular partial Nahm sums. These Nahm sums correspond to the matrices $A=M(a,b,c)$ where $(a,b,c)$ is one of
\begin{align*}
&(\alpha/2,1,0), (0,1,\alpha/2) ~~(\alpha>0); \quad (0,1/2,0); \quad (1/2,1/2,0), (0,1/2,1/2); \\
&(0,1,0); \quad (1,1,1);  \quad (1/2,1/2,1/2);  \quad (3/4,-1/2,1), ~~(1,-1/2,3/4); \\
& (1,-1,2),  (2,-1,1);   \quad (1,-1/2,1); \quad (3/2,1,2), (2,1,3/2); \quad  (2,1,1), (1,1,2); \\
& (1/2,-1/2,1), (1,-1/2,1/2); \quad (\alpha,1-\alpha,\alpha)~~ (\alpha>0); \quad (3/4,-1/4,3/4).
\end{align*}

There are some interesting problems to consider in the future. 

When the rank $r=1$, since Zagier found all modular rank one Nahm sums, it is natural to ask whether Theorem \ref{thm-rank1} exhaust all modular rank one partial Nahm sums.

When the rank $r\geq 2$, since a complete classification of modular rank $r$ Nahm sums is still far from reach, a complete classification of modular rank $r$ partial Nahm sums seems even more difficult.

By the way, it is not clear to us what kind of modular Nahm sums contain modular partial Nahm sums. In other words, when $f_{A,B,C}(q)$ is modular, can we always find lattice $L\subset \mathbb{Z}^r$ and vector $v\in \mathbb{Z}^r$ such that the partial Nahm sum $f_{A,B,C,v+L}(q)$ is modular? This seems not always true. For example, in the rank one case, $f_{2,0,-1/60}(q)$ and $f_{2,1,11/60}(q)$ are modular but we do not find partial Nahm sums $f_{2,B,C,v+L}(q)$ ($B=0,1$) which are modular. We give one more example in the rank two case. Cao and Wang \cite[Eq.\ (3.4)]{Cao-Wang} proved that
\begin{align}\label{eq-CW}
    \sum_{i,j\geq 0} \frac{q^{((i-j)^2+i+j)/2}}{(q;q)_i(q;q)_j}=\frac{(q^2;q^2)_\infty^2}{(q;q)_\infty^3}.
\end{align}
This implies that $f_{A,B,C}(q)$ is a modular form of weight $-1/2$ for $A=M(1,-1,1)$, $B=(1/2,1/2)^\mathrm{T}$ and $C=1/24$.
But we do not find any lattice coset $v+L$ with $L$ given in \eqref{eq-lattice} such that $f_{A,B,C,v+L}(q)$ is an infinite product. Along the process, we find the following parameterized identity:
\begin{align}
\sum_{i,j\ge0}\frac{q^{2i^2-2ij+\frac{1}{2}j^2+i+\frac{1}{2}j}u^j}{(q;q)_{2i+1}(q;q)_{j}}&=\frac{(-q;q)_\infty(-uq^2;q)_\infty}{(uq;q)_\infty}. \label{add-8}
\end{align}
This follows from
    \begin{align*}
        &\sum_{i,j\ge0}\frac{q^{2i^2-2ij+\frac{1}{2}j^2+i+\frac{1}{2}j}u^j}{(q;q)_{2i+1}(q;q)_{j}}=\frac{1}{2}\sum_{i,j\ge0}\frac{q^{\frac{1}{2}i^2-ij+\frac{1}{2}j^2-\frac{1}{2}i+\frac{3}{2}j}u^j(1-(-1)^i)}{(q;q)_{i}(q;q)_{j}} \\
        &=\frac{1}{2}\left(\sum_{j=0}^{\infty}\frac{q^{\frac{1}{2}j^2+\frac{3}{2}j}u^j}{(q;q)_{j}}\sum_{i=0}^{\infty}\frac{q^{\frac{1}{2}i^2-\frac{1}{2}i-ij}}{(q;q)_{i}}-\sum_{j=0}^{\infty}\frac{q^{\frac{1}{2}j^2+\frac{3}{2}j}u^j}{(q;q)_{j}}\sum_{i=0}^{\infty}\frac{(-1)^{i}q^{\frac{1}{2}i^2-\frac{1}{2}i-ij}}{(q;q)_{i}}\right)\\
        &=\frac{1}{2}\left(\sum_{j=0}^{\infty}\frac{q^{\frac{1}{2}j^2+\frac{3}{2}j}u^j(-q^{-j};q)_\infty}{(q;q)_{j}}-\sum_{j=0}^{\infty}\frac{q^{\frac{1}{2}j^2+\frac{3}{2}j}u^j(q^{-j};q)_\infty}{(q;q)_{j}}\right)\\
        &=(-q;q)_\infty\sum_{j=0}^{\infty}\frac{q^ju^j(-q;q)_{j}}{(q;q)_{j}}=\frac{(-q;q)_\infty(-uq^2;q)_\infty}{(uq;q)_\infty}.
    \end{align*}
Here for the last second equality we used \eqref{q-product-zero} and for the last equality we used \eqref{q-binomial}. But the right side of \eqref{add-8} is not modular for any $u=q^a$ ($a\in \mathbb{Q}$).

Furthermore, we are curious about the weights of modular generalized or partial Nahm sums. Zagier \cite{Zagier} proved that if the Nahm sum $f_{A,B,C}(q)$ with positive definite $A$ is modular, then its weight must be zero. In this paper we relax the restriction and only  require $A$ to be a symmetric nonzero matrix and the sum in \eqref{eq-full-Nahm} converge absolutely. In this setting, the weight is not always zero. For example, the weight of the modular Nahm sum in \eqref{eq-CW} is $-1/2$.  While the modular partial Nahm sums in Sections \ref{sec-zero} and \ref{sec-positive} all have weights zero, the modular partial Nahm sums in Section \ref{sec-negative} have weights $1/2$ (see Theorem \ref{thm-010}), $0$ or $-1/2$ (see Theorem \ref{thm-alpha10}). It would be very meaningful if one can find all possible weights of such modular generalized/partial Nahm sums.

Finally, as Nahm sums have nice applications in Lie algebras and conformal field theory,  we wonder whether there are algebraic interpretations or physical background for Nahm sums on lattice cosets.


\subsection*{Acknowledgements}
This work was supported by the National Key R\&D Program of China (Grant No.\ 2024YFA1014500).

\end{document}